\documentclass[12pt]{amsart}
\usepackage{ifthen,verbatim}
\usepackage{floatflt,graphicx,graphics}

\usepackage{tikz} 
\usepackage{mathrsfs}
\usepackage{amsmath,amssymb,amsthm}
\usepackage{mathtools}
\usepackage{multirow}

\numberwithin{equation}{section}
\nonstopmode
\numberwithin{equation}{section}
\theoremstyle{plain}
\newtheorem{theorem}[equation]{Theorem}
\newtheorem{thm}[equation]{Theorem}
\newtheorem{conjecture}[equation]{Conjecture}
\newtheorem{lemma}[equation]{Lemma}
\newtheorem{corollary}[equation]{Corollary}

\newtheorem{prop}[equation]{Proposition}

\theoremstyle{definition}

\newtheorem{remark}[equation]{Remark}
\newtheorem{nonsec}[equation]{}

\theoremstyle{remark}

\newcommand{\R}{\mathbb{R}}

\newcommand{\C}{\mathbb{C}}

\newcommand{\B}{\mathbb{B}}

\newcounter{alphabet}

\newcounter{minutes}
\divide\time by 60
\newcounter{hours}
\multiply\time by 60
\addtocounter{minutes}{-\time}

\begin{document}
\bibliographystyle{amsplain}
\title
{
Collinearity of points on Poincaré unit disk and Riemann sphere
}

\def\thefootnote{}
\footnotetext{
\texttt{\tiny File:~\jobname .tex,
          printed: \number\year-\number\month-\number\day,
          \thehours.\ifnum\theminutes<10{0}\fi\theminutes}
}
\makeatletter\def\thefootnote{\@arabic\c@footnote}\makeatother

\author[M. Fujimura]{Masayo Fujimura}
\author[O. Rainio]{Oona Rainio}
\author[M. Vuorinen]{Matti Vuorinen}

\begin{abstract}
We study certain points significant for the hyperbolic geometry of the unit disk. We give explicit formulas for the intersection points of the Euclidean lines and the stereographic projections of great circles of the Riemann sphere passing through these points. We prove several results related to collinearity of these intersection points, offer new ways to find the hyperbolic midpoint, and represent a formula for the chordal midpoint. The proofs utilize Gr\"obner bases from computer algebra for the solution of polynomial equations.
\end{abstract}
\maketitle

\noindent \textbf{Author information.}\\
Masayo Fujimura$^1$, email: \texttt{masayo@nda.ac.jp}, ORCID: 0000-0002-5837-8167\\
Oona Rainio$^2$, email: \texttt{ormrai@utu.fi}, ORCID: 0000-0002-7775-7656\\
Matti Vuorinen$^2$, email: \texttt{vuorinen@utu.fi}, ORCID: 0000-0002-1734-8228\\
1: Department of Mathematics, National Defense Academy of Japan, Yokosuka, Japan\\
2: Department of Mathematics and Statistics, University of Turku, FI-20014 Turku, Finland\\
\textbf{Funding.} The first author was partially supported by JSPS KAKENHI (Grant number: JP19K03531) and the second author was supported by Finnish Culture Foundation and Magnus Ehrnrooth Foundation.\\
\textbf{Data availability statement.} Not applicable, no new data was generated.\\
\textbf{Conflict of interest statement.} There is no conflict of interest.\\
\textbf{Acknowledgements.} We are grateful to the referees for their careful work and, in particular, the second part of Remark \ref{rmk_3.9} was suggested by one of the referees. We are thankful to Heikki Ruskeep\"a\"a for his help with Figure \ref{fig_cmp}.

\section{Introduction}

In recent decades, hyperbolic geometry and its many generalizations have found numerous applications to geometric function theory and analysis in Euclidean spaces $\mathbb{R}^n, n\ge 2,$ and in metric spaces \cite{hkv}. One of the reasons why the hyperbolic geometry is better suited to function theoretic applications than the Euclidean metric is its natural invariance properties under conformal mappings and M\"obius transformations. Conformal invariants play an important role in geometric function theory and potential theory.

Introduction to hyperbolic geometry is given for instance in \cite{b,ber,p01,r}. Many classical facts of Euclidean geometry have their counterparts in the hyperbolic geometry, usually in substantially different form, for instance the sum of angles of a triangle is smaller than $\pi.$ A systematic compilation
of trigonometric formulas for the three geometries: Euclidean, hyperbolic, and spherical geometries can be found in \cite{j}. Yet it is difficult to find explicit formulas even for simple geometric constructions such as the point of intersection of two hyperbolic lines, the hyperbolic midpoint of two points or the distance of a point from a hyperbolic line.
See also \cite{f,h}. 

For any two distinct arbitrary points in the unit disk non-collinear with the origin, we can create a group of six points by including the original points, the end points of the hyperbolic line passing through the two points, and their image points after an inversion in the unit circle. With these six points, we can define several intersection points by using such Euclidean lines that pass through certain pair of points chosen among them.
In the article \cite{vw1}, it was proven that five of these intersection points occur on the Euclidean line passing the origin and the hyperbolic midpoint of the original two points in the unit disk. An explicit formula for this hyperbolic midpoint, recently found in \cite{wvz}, see Theorem \ref{myhmidp} below,
has an important role in our work. 
  
We continue here the earlier work by finding explicit formulas for the five collinear intersection points. This problem is difficult to solve either by hand calculations or by using symbolic computation systems alone, but a combination of these two methods produces results when applied correctly. During the course of this work, we have used a symbolic computation program called Risa/Asir \cite{risa} to perform substitution calculations and reduction of fractions so that the resulting formulas are simple enough to be of use. These formulas of the intersection points are summarized in Table \ref{t1} in Section 3, which can be understood better by looking at the related geometry shown in Figure \ref{fig:kstuvw}. We also studied the same problem in the spherical geometry by computing the similar formulas for the stereographic projections of intersection points of the great circle that pass through the same points used to define Euclidean lines. These intersection points are also collinear, as proven in Theorem \ref{thm_5pointschordal}, and their definitions and formulas are presented in Figure \ref{fig_gcp} and Table \ref{tgc}.

The structure of this article is as follows. In Section 3, we give explicit formulas for the intersection points defined with Euclidean lines. In Section 4, we study the stereographic projections of intersection points of the great circles and prove that they are collinear not only with each other but also the earlier intersection points. In Section 5, we present several results related to the hyperbolic midpoint and also give a formula for the chordal midpoint of two points in the unit disk. We also formulate a conjecture about hyperbolic distances of certain intersection points.

\section{Preliminaries}

Define the complex conjugate of the point $z$ in the complex plane $\C$ as $\overline{z}={\rm Re}(z)-{\rm Im}(z)i$. Denote the $n$-dimensional unit ball by $\B^n$ and the unit sphere by $S^{n-1}$. The Euclidean line $L[a,b]$ passing through $a,b$ ($ a\neq b$) 
is given by
\begin{equation}
(\overline{a}-\overline{b})z-(a-b)\overline{z}
      =\overline{a}b-a\overline{b}.    
\end{equation}
The reflection of a point $x$ in the line $L[a,b]$ is given by 
\begin{equation}\label{myrefl}
w(x) = \frac{a-b}{ \overline{a}- \overline{b}}\, \overline{x} -\frac{a \overline{b}- \overline{a}b}{ \overline{a} -  \overline{b}}.
\end{equation}    
If the two lines $ L[a,b] $ and $ L[c,d] $ are non-parallel, their unique intersection point is given by
\begin{equation}\label{eq:lis}
  {\rm LIS}[a,b,c,d]=
   \frac{(\overline{a}b-a\overline{b})(c-d)
          -(\overline{c}d-c\overline{d})(a-b)}
        {(\overline{a}-\overline{b})(c-d)
          -(\overline{c}-\overline{d})(a-b)}.
\end{equation}
This formula can be found in Exercise 4.3(1) on p.\,57 and the solution on p.\,373 of \cite{hkv}.

\begin{prop}\label{LISprop}
Let $a,b \in \mathbb{C}$ with $|a|\neq |b|, |a||b| \neq 1.$ Then
\begin{align*}
(1)& \quad \quad {\rm LIS}[a,b, -1/\overline{a}, -1/\overline{b}]= 
\frac{b(1+|a|^2)-a(1+|b|^2)}{|a|^2 -|b|^2},\\
(2)& \quad \quad {\rm LIS}[a,b, 1/\overline{a}, 1/\overline{b}]= 
\frac{a(1-|b|^2)-b( 1-|a|^2)}{|a|^2 -|b|^2},\\
(3)& \quad \quad {\rm LIS}[a, 1/\overline{b},b, 1/\overline{a}]= 
\frac{a(1-|b|^2)+b(1-|a|^2)}{1-|a|^2 |b|^2}\, ,\\
(4)& \quad \quad {\rm LIS}[a, -1/\overline{b},b, -1/\overline{a}]= 
\frac{a(1+|b|^2)+b(1+|a|^2)}{1-|a|^2 |b|^2}\,.    
\end{align*}
\end{prop}
\begin{proof}
The proof follows from \eqref{eq:lis}.
\end{proof}

\begin{prop}\label{UCLISprop}
Let $a,b \in \mathbb{B}^2$ be points non-collinear with $0$ and
$c={\rm LIS}[a,b,0, i(a-b)].$ Then $L[a,b]\cap S^1 =\{a_1,b_1\}$
where
$$  a_1=c-i \frac{c}{|c|}\sqrt{1-|c|^2}, \quad b_1=c+i \frac{c}{|c|}\sqrt{1-|c|^2},$$
and $a_1,a,b,b_1$ are ordered in such a way that $|a_1-a|<|a_1-b|.$
\end{prop}
\begin{proof}
The proof follows from \eqref{eq:lis}.
\end{proof}

\begin{prop}\cite[4., p.13]{p01}\label{prop_conjIntF}
If $a,b,c,d$ are four distinct complex points on the unit circle $S^1$, then the complex conjugate $\overline{f}$ of the intersection point $f$ of the lines $L[a,c]$ and $L[b,d]$ is
\begin{align*}
\overline{f}=\frac{a+c-b-d}{ac-bd},
\end{align*}
assuming that these two lines intersect.
\end{prop}
\begin{proof}
Because $\overline{z}=1/z$ for $z\in\{a,b,c,d\}$, the proof follows from \eqref{eq:lis}.
\end{proof}

Let $C[a,b,c]$  be the  circle  through distinct non-collinear points $\,a,\,b,\, c\,.$
The formula \eqref{eq:lis} gives easily the formula for the center $m(a,b,c)$ of $C[a,b,c]\,.$
For instance, we can find two points on the bisecting normal to the side $[a,b]$ and another
two points on the bisecting normal to the side $[a,c]$ and then apply \eqref{eq:lis}  to get
$m(a,b,c)\,.$
In this way we see that the center $m(a,b,c)$ of  $C[a,b,c]$  is
\begin{equation}\label{mfun}
m(a,b,c)=\frac{ |a|^2(b-c) +  |b|^2(c-a) +  |c|^2(a-b) }
{a(\overline{c}-\overline{b}) +b(\overline{a}-\overline{c})+ c(\overline{b}-\overline{a})}\,.
\end{equation}
The following special cases of  \eqref{mfun} will be very useful:
\begin{equation}\label{geoLines}
\begin{cases}
{\displaystyle
{m(a,1/\overline{a},b)=  (-a + b + a b (\overline{a} - \overline{b}))/d }}\,;\quad d=b \overline{a}- a \overline{b},&\\
{\displaystyle {m(a,-1/\overline{a},b) = (a - b + a b (\overline{a} -\overline{b}))/d\,.}}&
\end{cases}
\end{equation}

A \emph{M\"obius transformation} is a mapping of the form
$$z \mapsto \frac{az+b}{cz+d}\,, \quad a,b,c,d,z \in {\mathbb C}\,,\quad ad-bc\neq 0\,.$$ The most important feature of M\"obius transformations is that they always preserve the angle magnitude and, because of this, they map every Euclidean line or circle onto either a line or a circle. The special M\"obius transformation
\begin{equation}\label{myT}
T_a(z) = \frac{z-a}{1- \overline{a}z}\,, \quad a \in \B^2\setminus \{0\}\,,
\end{equation}
maps the unit disk $\mathbb{B}^2$ onto itself with $T_a(a) =0$, $T_a( \pm a/|a|)= \pm a/|a|$.

In the extended real space $\overline{\R}^n$, the {\it spherical (chordal) metric} $q$ is defined as \cite[(3.6), p. 29]{hkv}
\begin{equation}\label{1.15.}
  \begin{cases}
     {\displaystyle
			q(x,y)=\frac{|x-y|}{\sqrt{1+|x|^2}\;\sqrt{1+|y|^2}}}\;;\,\;x\ne\infty\ne y\;,&\\
	{\displaystyle 
	q(x,\infty)=\frac{1}{\sqrt{1+|x|^2}}}\;\,.&
  \end{cases}
\end{equation}
By using this metric, we can define the \emph{absolute ratio} of any four distinct points $a,b,c,d\in\overline{\R}^n$ as \cite[(3.10), p. 33]{hkv}, \cite{b}
\begin{align*}
|a,b,c,d|=\frac{q(a,c)q(b,d)}{q(a,b)q(c,d)},
\end{align*}
where the spherical distances can be replaced with the Euclidean distances if and only if $\infty\notin\{a,b,c,d\}$. The absolute ratio is a very useful tool in complex analysis because it is invariant under M\"obius transformations, which means that $|a,b,c,d|=|h(a),h(b),h(c),h(d)|$ for any M\"obius transformation $h$.

Next, let us recall some basic formulas and notation for hyperbolic geometry from \cite{b}. The hyperbolic metric $\rho_{\B^n}$ in the unit ball $\B^n$ is defined by
\begin{equation}\label{myrho}
{\rm sh}\frac{\rho_{\B^n}(x,y)}{2}= \frac{|x-y|}{\sqrt{(1-|x|^2)(1-|y|^2)}}\,,
\end{equation}
where sh stands for the hyperbolic sine function. Correspondingly, we denote the hyperbolic cosine and tangent by ch and th.
In the two-dimensional case, the formula \eqref{myrho} is equivalent to
\begin{align}
{\rm th}\frac{\rho_{\B^2}(x,y)}{2}=\frac{|x-y|}{|1-x\overline{y}|}=\frac{|x-y|}{A[x,y]}, 
\end{align}
where $A[x,y]$ is the Ahlfors bracket defined as \cite[7.37]{avv} 
\begin{equation}\label{myahl}
A[x,y]=|1-x\overline{y}|=\sqrt{|x-y|^2+(1-|x|^2)(1-|y|^2)}.
\end{equation}

The hyperbolic line through $a,b\in\B^2$ is denoted by $J^*[a,b]$. If $a,b$ are collinear with the origin, the line $J^*[a,b]$ is a diameter of the unit disk and, otherwise, an arc of the circle that passes through $a,b$ and is orthogonal to the unit circle. For $a,b\in\B^2\setminus\{0\}$, let 
\begin{equation}\label{epdef}
ep(a,b)=T_{-b}(T_b(a)/|T_b(a)|).
\end{equation}
This formula defines the end points on the unit circle for the hyperbolic line $J^*[a,b]$. The hyperbolic metric also satisfies
\begin{equation}\label{epdefrho}
\rho_{\B^2}(x,y)=\log |ep(a,b),a,b,ep(b,a)|,
\end{equation}
which shows, in particular, that the hyperbolic metric is invariant under M\"obius transformations of the unit disk onto itself. In fact, the hyperbolic metric is invariant under all conformal mappings and the M\"obius transformations are one subclass of conformal mappings.

\begin{thm}\label{myhmidp}\cite[Thm 1.4, p. 126]{wvz}
For given $x,y\in\B^2$, the hyperbolic midpoint $z\in\B^2$ with
$\rho_{\B^2}(x,z)=\rho_{\B^2}(y,z)=\rho_{\B^2}(x,y)/2$ is given by
\begin{equation}\label{myzformua}
z=\frac{y(1-|x|^2) + x(1-|y|^2)}{1-|x|^2|y|^2 + A[x,y] \sqrt{(1-|x|^2)(1-|y|^2)}}.
\end{equation}
\end{thm}

The extended real plane $\overline{\R}^2$ is 
identified with the Riemann sphere via the stereographic projection. Denote $e_3=(0,0,1)$. The {\it stereographic
projection } \ $\ \pi:\overline{\mathbb{R}}^2\to S^2(\frac{1}{2} e_3,\frac{1}{2} )$
is defined by \cite[(3.4), p. 28]{hkv}
\begin{equation}\label{1.13.}
  \pi(x)= e_3+\frac{x-e_3}{ |x-e_3|^2} \,,\;x\in {\mathbb{R}}^2\;;\;\,
  \pi(\infty)=e_3\;.
\end{equation}
Then $\pi$ is the restriction to $\overline{\mathbb{R}}^2$ of the inversion in
$S^2(e_3,1)$. 
Because  $f^{-1}=f$  for every inversion  $f$, it follows that  $\pi$
maps the ``Riemann sphere'' $S^2(\frac{1}{2} e_3,\frac{1}{2} )$
onto  $\overline{\mathbb{R}}^2$. By applying the stereographic projection $\pi$ from \eqref{1.13.}, the definition of the chordal metric $q$ in $\overline{\mathbb{R}}^2$ can be written as \cite[(3.5), p. 29]{hkv}
\begin{equation}\label{1.14.}
   q(x,y)=|\pi(x)-\pi(y)|\;;\;\, x,\,y\in  \overline{\mathbb{R}}^2\;,
\end{equation}
which is equivalent to \eqref{1.15.} for $n=2$.

The {\it antipodal point} $\tilde{a} = -1/\overline{a}$ satisfies
$q(a,\tilde{a})=1$ \cite[p. 29]{hkv}. Thus, the great circle through two given points
$a, b \in\mathbb{C}$ has the center $m(a,\tilde{a}, b)$ given by
\eqref{geoLines} and its Euclidean radius is of course $|a- m(a,\tilde{a}, b)|$
while its chordal radius is $1/\sqrt{2}.$ For instance, if 
$a\in {\mathbb{R}^2}\setminus \{0\},$ then $q(a,\pm ia/|a|) = 1/\sqrt{2}.$ Note that if $a,b\in\C$ are collinear with the origin, then $m(a,\tilde{a}, b)$ is not well-defined, and the stereographic projection of the great circle through $a$ and $b$ on the plane $\C$ is a Euclidean line $L(a,b)$.
For further information, see \cite[Thm 18.4.2]{ber}.

Hereafter, we will discuss mainly geometry on the complex plane.

For a point $ z\in\mathbb{C} $, the stereographic projection $ \pi $ in \eqref{1.13.} is written as
\begin{align}\label{riePoint}
      \pi\,:\, z\mapsto \Big(\frac{\mbox{Re}\, z}{1+|z|^2},\frac{\mbox{Im}\, z}{1+|z|^2},
      \frac{|z|^2}{1+|z|^2}\Big).
\end{align}
Conversely, the point $ z\in\mathbb{C} $ corresponding to the point
$ (\xi,\eta,\zeta) $ on the Riemann sphere can be written as
$$
    z=\frac{\xi}{1-\zeta}+i\frac{\eta}{1-\zeta},
     \quad \mbox{where}\quad\
     \xi^2+\eta^2+\big(\zeta-\frac12\big)^2=\big(\frac14\big).
$$

Assume that points $ \pi(a)$, $ \pi(b)$, $ \pi(c)$, and $ \pi(d)$ are not on the same great circle. The intersection point of stereographic projections of two great circles passing through $ \pi(a)$, $ \pi(b)$ and $ \pi(c)$, $ \pi(d)$ inside the closed unit disk is denoted as
\begin{align}\label{equ_gcis}
{\rm GCIS}[a,b,c,d]=\{z\in\overline{\B}^2\,:\,F(z;a,b,c,d)=0\},    
\end{align}
where the function $F$ is as follows:
\begin{equation*}
\begin{aligned}
F(z;a,b,c,d)\equiv&   \big((\overline{a}\overline{b}(a-b)+\overline{a}-\overline{b})
     (c\overline{d}-\overline{c}d)-(a\overline{b}-\overline{a}b)
     (\overline{c}\overline{d}(c-d)+\overline{c}-\overline{d})\big)z^2 \\
  & \ 
    +\big(-(1-\overline{a}\overline{b}cd)(a-b)
     (\overline{c}-\overline{d})+(1-ab\overline{c}\overline{d})
     (\overline{a}-\overline{b})(c-d)\\
  & \   +(a-b)(c-d)(\overline{a}\overline{b}-\overline{c}\overline{d})
     -(\overline{a}-\overline{b})(\overline{c}-\overline{d})(ab-cd)\big)z\\
  & \ -(c\overline{d}-\overline{c}d)(ab(\overline{a}-\overline{b})+(a-b))
    +(a\overline{b}-\overline{a}b)(cd(\overline{c}-\overline{d})+(c-d)).
\end{aligned}    
\end{equation*}
This result follows from the definition of great circles passing through two points and the formula of intersection points of two Euclidean circles. See also the proof of Theorem \ref{thm_cmp} and Appendix.

\section{Five collinear intersection points}

In this section, we study the five collinear intersection points of the Euclidean lines passing through such points that are defined in terms of two given points in the unit disk.

Fix first two distinct points $a,b\in\mathbb{B}^2\backslash\{0\}$ so that they are non-collinear with the origin. 
Let $ a^{\ast}$ and $ b^{\ast} $ be the reflections of $ a $ and $ b $ with respect to the
unit circle, respectively.
Then, 
\begin{equation}\label{eq:ab-ast}
 a^{\ast}=\frac{1}{\overline{a}}, \quad  b^{\ast}=\frac{1}{\overline{b}}. 
\end{equation}
Let $ a_{\ast} $  and $ b_{\ast} $ be the intersection points of
the unit circle and the hyperbolic line passing through $ a $ and $ b $ that
are near $ a $ and $ b $, respectively.
These points can be found with \eqref{epdef}, which yields the formulas
\begin{equation}\label{eq:ab-end}
  a_{\ast}=\frac{b(1-a\overline{b})|a-b|+(a-b)|1-a\overline{b}|}
              {(1-a\overline{b})|a-b|+\overline{b}(a-b)|1-a\overline{b}|},
\quad
  b_{\ast}=\frac{a(1-\overline{a}{b})|b-a|+(b-a)|1-\overline{a}{b}|}
              {(1-\overline{a}{b})|b-a|+\overline{a}(b-a)|1-\overline{a}{b}|}.
\end{equation}
In the notation of \eqref{epdef}, $a_*=ep(a,b)$ and $b_*=ep(b,a)$.

Fix then the intersection points
\begin{equation}\label{equ_kstuv}
\begin{aligned}
  k &= {\rm LIS}[a_{\ast},a^{\ast},b_{\ast},b^{\ast}],&  
  s &= {\rm LIS}[a,b_{\ast},b,a_{\ast}],&
  t &= {\rm LIS}[a_{\ast},b^{\ast},b_{\ast},a^{\ast}],\\
  u &= {\rm LIS}[a,b^{\ast},b,a^{\ast}],& 
  v &= {\rm LIS}[a,a_{\ast},b,b_{\ast}].&
\end{aligned}    
\end{equation}
Furthermore, let $m$ be the hyperbolic midpoint of $a$ and $b$. See Figure \ref{fig:kstuvw}. 

\begin{figure}
    \centering
    \begin{tikzpicture}[scale=3]
    \draw (-1.396,-1.145) -- (2,0);
    \draw (-1.396,-1.145) circle (0.3mm);
    \node[scale=1.3] at (-1.396,-0.98) {$k$};
    \draw (2,0) circle (0.3mm);
    \node[scale=1.3] at (1.97,0.15) {$a^*$};
    \draw (-1.396,-1.145) -- (0.771,1.202);
    \draw (0.771,1.202) circle (0.3mm);
    \node[scale=1.3] at (0.771,1.33) {$b^*$};
    \draw (0.5,0) -- (0.474,0.880);
    \draw (0.5,0) circle (0.3mm);
    \node[scale=1.3] at (0.45,-0.1) {$a$};
    \draw (0.474,0.880) circle (0.3mm);
    \node[scale=1.3] at (0.43,1.05) {$b_*$};
    \draw (0.378,0.589) -- (0.933,-0.359);
    \draw (0.378,0.589) circle (0.3mm);
    \node[scale=1.3] at (0.28,0.57) {$b$};
    \draw (0.933,-0.359) circle (0.3mm);
    \node[scale=1.3] at (1,-0.48) {$a_*$};
    \draw (0.771,1.202) -- (0.933,-0.359);
    \draw (0.474,0.880) -- (2,0);
    \draw (0.5,0) -- (0.771,1.202);
    \draw (0.378,0.589) -- (2,0);
    \draw (0.251,0.206) -- (0.933,-0.359);
    \draw (0.251,0.206) -- (0.474,0.880);
    \draw (0,0) circle (1cm);
    \draw (0,0) circle (0.3mm);
    \node[scale=1.3] at (0,0.15) {$0$};
    \draw[dashed] (1.25,0.462) circle (0.880cm);
    \draw (0.251,0.206) circle (0.3mm);
    \node[scale=1.3] at (0.18,0.28) {$v$};
    \draw (0.381,0.313) circle (0.3mm);
    \node[scale=1.3] at (0.4,0.2) {$m$};
    \draw (0.488,0.400) circle (0.3mm);
    \node[scale=1.3] at (0.4,0.43) {$s$};
    \draw (0.613,0.503) circle (0.3mm);
    \node[scale=1.3] at (0.56,0.6) {$u$};
    \draw (0.826,0.677) circle (0.3mm);
    \node[scale=1.3] at (0.87,0.81) {$t$};
    \draw (-1.6,-1.312) -- (2.3,1.886);
    \end{tikzpicture}
    \caption{The points of Theorem \ref{thm_5pointshyp} for $a=0.5$ and $b=0.7e^i$. The line passing through $k$ and the origin contains the points $k,0,v,m,s,u,t$ in this order. The dashed circle passing through $b^*,b_*,b,m,a,a_*,a^*$ is orthogonal to the unit circle.}
\label{fig:kstuvw}
\end{figure}
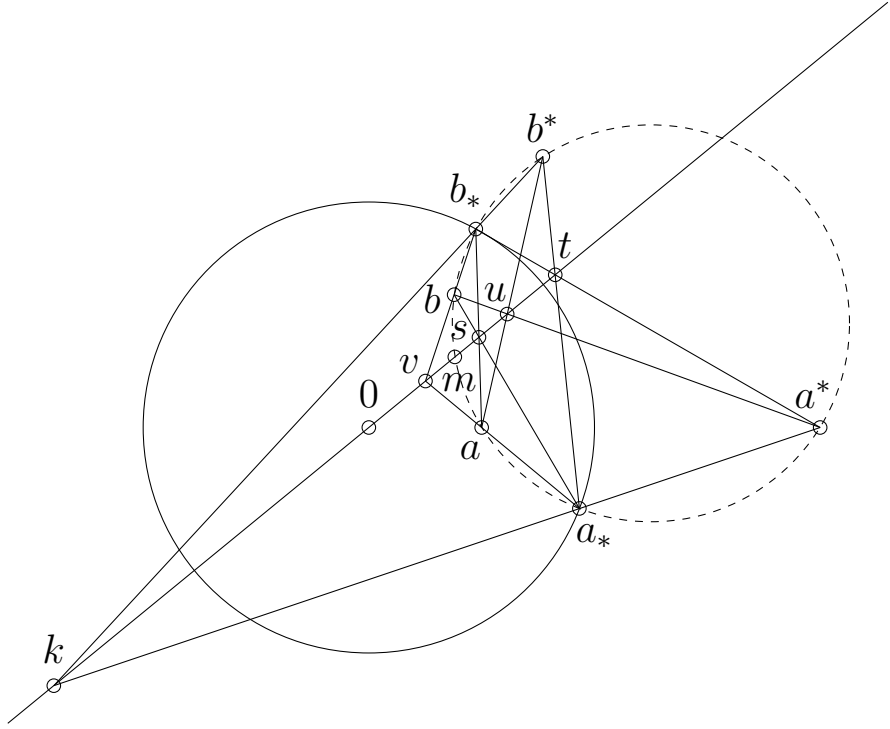

The following result was obtained in \cite{vw1} by using the properties of M\"obius transformations without finding the explicit formulas for the intersection points.

\begin{theorem}\label{thm_5pointshyp}\cite[Lemma 4.6, p. 282]{vw1}
For $a,b\in\B^2$, the five points $k,s,t,u,v$ fixed as in \eqref{equ_kstuv}, the hyperbolic midpoint $m$ of $a$ and $b$, and the origin are all collinear. 
\end{theorem}

The exact coordinates of the intersection points are given as follows.

\begin{theorem}\label{thm_formulas}
The explicit formulas of the points $k,s,t,u,v$ in \eqref{equ_kstuv} are
\begin{align*}
k&=\big(a(1-|b|^2)+b(1-|a|^2)\big)
       \dfrac{|a-b|-|1-\overline{a}b|}{(1-|ab|^2)|a-b|
         +\big(2|ab|^2-(|a|^2+|b|^2)\big)|1-\overline{a}b|},\\
s&=\dfrac{a(1-|b|^2)+b(1-|a|^2)}
            {2-(a\overline{b}+b\overline{a})-|a-b||1-\overline{a}b|},\\
t&=\frac{a(1-|b|^2)+b(1-|a|^2)}
       {(a\overline{b}+\overline{a}b)-2|ab|^2+|a-b||1-\overline{a}b|},\\
u&=\dfrac{a(1-|b|^2)+b(1-|a|^2)}{1-|ab|^2},\\
v&=\big(a(1-|b|^2)+b(1-|a|^2)\big)
      \dfrac{|1-\overline{a}b|-|a-b|}
           {\big(2-(|a|^2+|b|^2)\big)|1-\overline{a}b|-(1-|ab|^2)|a-b|}.
\end{align*}
\end{theorem}
\begin{proof}
The point $u$ is already obtained in \cite[(4.8), p.282]{vw1}. By fixing $m_{ab}=|a-b|$ and $m_1=|1-a\overline{b}|$,
$a_{\ast}$ can be written as
\begin{equation}\label{eq:aast}
a_{\ast}=\frac{b(1-a\overline{b})m_{ab}+(a-b)m_1}
              {(1-a\overline{b})m_{ab}+\overline{b}(a-b)m_1}.
\end{equation}
As $|a-b|=|b-a|$ and $|1-\overline{a}b|=|1-a\overline{b}|$,
\begin{equation}\label{eq:bast}
b_{\ast}=\frac{a(1-\overline{a}{b})m_{ab}-(a-b)m_1}
              {(1-\overline{a}{b})m_{ab}-\overline{a}(a-b)m_1}.
\end{equation}
By substituting \eqref{eq:ab-ast}, \eqref{eq:aast}, and \eqref{eq:bast} into
\eqref{eq:lis}, we have
$$
t={\rm LIS}[a_{\ast},b^{\ast},b_{\ast},a^{\ast}]
    = \frac{\big(a(1-b\overline{b})+b(1-a\overline{a})\big)m_1}{(1-a\overline{b})(1-\overline{a}b)m_{ab}
         +(a\overline{b}+\overline{a}b-2ab\overline{a}\overline{b})m_1}.
$$
The substitution calculations and reduction of fractions here can be obtained by using the symbolic computation system Risa/Asir as the calculation of substitutions and reduction of fractions is just an algebraic process.

From $ (1-a\overline{b})(1-\overline{a}b)=m_1^2$, it follows that
\begin{align*}
  t &= \frac{\big(a(1-b\overline{b})+b(1-a\overline{a})\big)m_1}
     {m_1^2m_{ab}+(a\overline{b}+\overline{a}b-2ab\overline{a}\overline{b})m_1}
     = \frac{a(1-b\overline{b})+b(1-a\overline{a})}
        {(a\overline{b}+\overline{a}b)-2ab\overline{a}\overline{b}+m_1m_{ab}}\\
    &= \frac{a(1-|b|^2)+b(1-|a|^2)}
       {(a\overline{b}+\overline{a}b)-2|ab|^2+|a-b||1-\overline{a}b|}.
\end{align*}

By using the same method as above for the points $s,k$, and $v$, we have
\begin{align*}
s &={\rm LIS}[a,b_{\ast},b,a_{\ast}]\\
  &=\frac{\big(a(1-b\overline{b})+b(1-a\overline{a})\big)m_{ab}}
         {(2-a\overline{b}-\overline{a}b)m_{ab}
            -(a-b)(\overline{a}-\overline{b})m_1}
   =\frac{\big(a(1-b\overline{b})+b(1-a\overline{a})\big)m_{ab}}
         {(2-a\overline{b}-\overline{a}b)m_{ab}
            -m_{ab}^2m_1}\\
   &=\frac{a(1-b\overline{b})+b(1-a\overline{a})}
         {(2-a\overline{b}-\overline{a}b)-m_{ab}m_1}
    =\dfrac{a(1-|b|^2)+b(1-|a|^2)}
            {2-(a\overline{b}+b\overline{a})-|a-b||1-\overline{a}b|},
\end{align*}

\begin{align*}
  k &={\rm LIS}[a_{\ast},a^{\ast},b_{\ast},b^{\ast}]\\
    &=\frac{\big(a(1-b\overline{b})+b(1-a\overline{a})\big)(m_{ab}-m_1)}
       {(1-ab\overline{a}\overline{b})m_{ab}
           +(2ab\overline{a}\overline{b}-(a\overline{a}+b\overline{b}))m_1}\\
    &=\big(a(1-|b|^2)+b(1-|a|^2)\big)
       \dfrac{|a-b|-|1-\overline{a}b|}{(1-|ab|^2)|a-b|
         +\big(2|ab|^2-(|a|^2+|b|^2)\big)|1-\overline{a}b|},
\end{align*}

and

\begin{align*}
  v &= {\rm LIS}[a,a_{\ast},b,b_{\ast}]\\ 
    &=\frac{\big(a(1-b\overline{b})+b(1-a\overline{a})\big)
           \big((1-a\overline{b})(1-\overline{a}b)m_{ab}
                 -(a-b)(\overline{a}-\overline{b})m_1\big)}
     {(1-a\overline{b})(1-\overline{a}b)(2-a\overline{a}-b\overline{b})m_{ab}
         -(a-b)(\overline{a}-\overline{b})(1-ab\overline{a}\overline{b})m_1}\\
    &=\big(a(1-b\overline{b})+b(1-a\overline{a})\big)
       \frac{m_1^2m_{ab}-m_{ab}^2m_1}
           {m_1^2(2-a\overline{a}-b\overline{b})m_{ab}
                 -m_{ab}^2(1-ab\overline{a}\overline{b})m_1}\\
   &=\big(a(1-|b|^2)+b(1-|a|^2)\big)
      \dfrac{|1-\overline{a}b|-|a-b|}
           {\big(2-(|a|^2+|b|^2)\big)|1-\overline{a}b|-(1-|ab|^2)|a-b|}.
\end{align*}
\end{proof}

Each of these five points $k,s,t,u$, and $v$ has the form $r\big(a(1-|b|^2)+b(1-|a|^2)\big)$ for some real number $r$. From the formula \eqref{myzformua} of Theorem \ref{myhmidp}, we see trivially  that the hyperbolic midpoint $m$ also has this same property. Therefore, all these six points are on the line 
passing through the origin and the point $a(1-|b|^2)+b(1-|a|^2)$. This would also prove the collinearity stated in Theorem \ref{thm_5pointshyp}. Additionally, we can write these points in the form $c_jH$ for $j=1,...,6$ when $H=a(1-|b|^2)+b(1-|a|^2)$ and the constants $c_j$ are as in Table \ref{t1}.

\begin{table}[ht]
    \centering
    \begin{tabular}{|l|l|l|l|}
         \hline
         $j$ & Point & Definition & $c_j$\\
         \hline
         \multirow{2}{*}{1} & \multirow{2}{*}{$k$} & \multirow{2}{*}{${\rm LIS}[a_{\ast},a^{\ast},b_{\ast},b^{\ast}]$} & \multirow{3}{*}{$\dfrac{|a-b|-|1-\overline{a}b|}{(1-|ab|^2)|a-b|
         +\big(2|ab|^2-(|a|^2+|b|^2)\big)|1-\overline{a}b|}$}\\
         &&&\\
         &&&\\
         \hline
         \multirow{2}{*}{2} & \multirow{2}{*}{$s$} & \multirow{2}{*}{${\rm LIS}[a,b_{\ast},b,a_{\ast}]$} & \multirow{3}{*}{$\dfrac{1}
            {2-(a\overline{b}+b\overline{a})-|a-b||1-\overline{a}b|}$}\\
         &&&\\
         &&&\\
         \hline
         \multirow{2}{*}{3} & \multirow{2}{*}{$t$} & \multirow{2}{*}{${\rm LIS}[a_{\ast},b^{\ast},b_{\ast},a^{\ast}]$} & \multirow{3}{*}{$\dfrac{1}{(a\overline{b}+\overline{a}b)-2|ab|^2+|a-b||1-\overline{a}b|}$}\\
         &&&\\
         &&&\\
         \hline
         \multirow{2}{*}{4} & \multirow{2}{*}{$u$} & \multirow{2}{*}{${\rm LIS}[a,b^{\ast},b,a^{\ast}]$} & \multirow{3}{*}{$\dfrac{1}{1-|ab|^2}$}\\
         &&&\\
         &&&\\
         \hline
         \multirow{2}{*}{5} & \multirow{2}{*}{$v$} & \multirow{2}{*}{${\rm LIS}[a,a_{\ast},b,b_{\ast}]$} & \multirow{3}{*}{$\dfrac{|1-\overline{a}b|-|a-b|}
           {\big(2-(|a|^2+|b|^2)\big)|1-\overline{a}b|-(1-|ab|^2)|a-b|}$}\\
         &&&\\
         &&&\\
         \hline
         \multirow{2}{*}{6} & \multirow{2}{*}{$m$} & Hyperbolic
          & \multirow{3}{*}{$\dfrac{1}{1-|a|^2|b|^2 + A[a,b] \sqrt{(1-|a|^2)(1-|b|^2)}}$}\\
         && midpoint of &\\
         && $a$ and $b$ &\\
         \hline
    \end{tabular}
    \medskip
    \caption{The constants $c_j$, $j=1,...,6$, with which the formulas of the points $k,s,t,u,v$ from \eqref{equ_kstuv} and the hyperbolic midpoint $m$ of $a,b$ can be written in the form $c_jH$ where $H=a(1-|b|^2)+b(1-|a|^2)$.}
    \label{t1}
\end{table}

\begin{remark}
The origin coincides with the intersection point ${\rm LIS}[a,a^{\ast},b,b^{\ast}]$ because $a^{\ast}$ and $b^{\ast}$ are the reflections of $a$ and $b$, respectively.
\end{remark}

\begin{remark}\label{rmk_3.9}
(1) It is clear from Pascal's theorem (can be found e.g. in \cite[Sect. 3]{g11}) that the three points $s,t,u$ are collinear.\newline 
(2) The triangles $sab$ and $tb^*a^*$ are perspectives from the point $u$. According to the Desargues theorem \cite[Sect. 3]{g11}, the point $M=L(a,b)\cap L(a^*,b^*)$ is located on the line $L(a_*,b_*)$. From Desargues's theorem, concerning suitable
triangles, the other correspondences can also be proved, the only exception being the point $m$. Let $m^*$ be the point of intersection of the dashed circle with the axis. The central-axial collinearity given by the center $M$, the axis $Ou$ and the pair of points $\{a,b\}$ sends the cross-section $(a_*,a,m^*,b_*)$ to the cross-section $(b_*,b,m^*,a_*)$, therefore $m^*=m$, i.e. $m$ is also a point of the corresponding line.
\end{remark}

\section{Collinearity results for great circles}

In this section, we define points similar to $k,s,t,u,v$ in \eqref{equ_kstuv} but by using the stereographic projections of intersection points of great circles instead the intersection points of Euclidean lines. The new points are also collinear, as stated in Theorem \ref{thm_5pointschordal}, which is the corresponding version of Theorems \ref{thm_5pointshyp} and \ref{thm_formulas}. In fact, we show in Corollary \ref{cor_11points} that all the resulting eleven points are collinear with the origin. At the end of this section, another collinearity result, Theorem \ref{thm_pqc}, is also proved.


For such two points $a,b\in\mathbb{B}^2$ that are non-collinear with the origin, let $a^*,b^*,a_*,b_*$ be as in \eqref{eq:ab-ast} and \eqref{eq:ab-end}. Fix then $k_c,s_c,t_c,u_c,v_c$ as the intersection points of the stereographic projections of great circles so that 
\begin{equation}\label{equ_kstuv_c}
\begin{aligned}
k_c&={\rm GCIS}[a_{\ast},a^{\ast},b_{\ast},b^{\ast}],\quad 
s_c={\rm GCIS}[a,b_{\ast},b,a_{\ast}],\quad
t_c={\rm GCIS}[a_{\ast},b^{\ast},b_{\ast},a^{\ast}],\\
u_c&\in{\rm GCIS}[a,b^{\ast},b,a^{\ast}],\quad
v_c={\rm GCIS}[a,a_{\ast},b,b_{\ast}].
\end{aligned}    
\end{equation}
In other words, $k_c,s_c,t_c,u_c,v_c$ are as the points $k,s,t,u,v$ in \eqref{equ_kstuv}, except they are the stereographic projections of intersection points of great circles while $k,s,t,u,v$ are simply the intersection points of Euclidean lines. Note that there are two possible choices for $u_c$ as there are two intersection points in ${\rm GCIS}[a,b^{\ast},b,a^{\ast}]$ that both lie on the unit circle. 

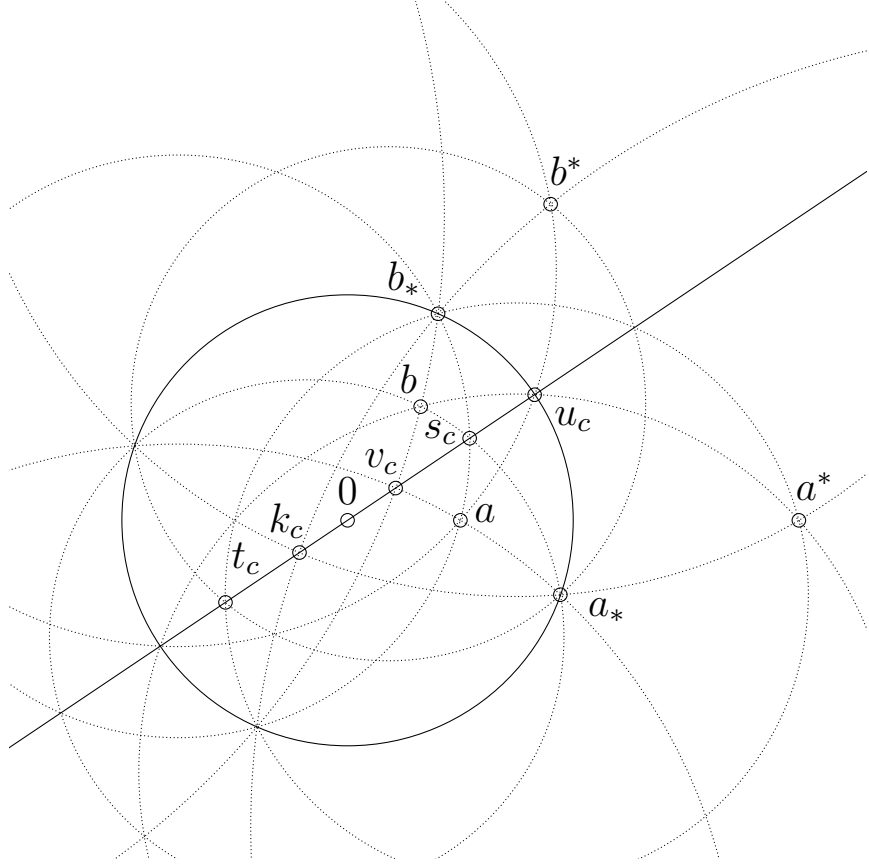
\begin{figure}
    \centering
    \begin{tikzpicture}[scale=3]
    \clip (-1.5,-1.5) rectangle (2.3,2.3);
    \draw (0.5,0) circle (0.3mm);
    \node[scale=1.3] at (0.61,0.03) {$a$};
    \draw (0.324,0.504) circle (0.3mm);
    \node[scale=1.3] at (0.27,0.63) {$b$};
    \draw (2,0) circle (0.3mm);
    \node[scale=1.3] at (2.07,0.15) {$a^*$};
    \draw (0.900,1.402) circle (0.3mm);
    \node[scale=1.3] at (0.97,1.55) {$b^*$};
    \draw (0.943,-0.330) circle (0.3mm);
    \node[scale=1.3] at (1.15,-0.4) {$a_*$};
    \draw (0.401,0.916) circle (0.3mm);
    \node[scale=1.3] at (0.25,1.07) {$b_*$};
    \draw (0,0) circle (0.3mm);
    \node[scale=1.3] at (0,0.13) {$0$};
    \draw (0,0) circle (1cm);
    \draw[densely dotted] (0.750,2.142) circle (2.480cm);
    \draw[densely dotted] (3.105,-1.360) circle (3.534cm);
    \draw[densely dotted] (-0.750,0.328) circle (1.292cm);
    \draw[densely dotted] (-0.181,-0.517) circle (1.140cm);
    \draw[densely dotted] (0.181,0.517) circle (1.140cm);
    \draw[densely dotted] (0.750,-0.328) circle (1.292cm);
    \draw[densely dotted] (-0.750,1.115) circle (1.675cm);
    \draw[densely dotted] (0.750,-1.115) circle (1.675cm);
    \draw[densely dotted] (-0.750,-2.142) circle (2.480cm);
    \draw[densely dotted] (-3.105,1.360) circle (3.534cm);
    \draw (-0.213,-0.143) circle (0.3mm);
    \node[scale=1.3] at (-0.27,0) {$k_c$};
    \draw (0.541,0.364) circle (0.3mm);
    \node[scale=1.3] at (0.41,0.39) {$s_c$};
    \draw (-0.541,-0.364) circle (0.3mm);
    \node[scale=1.3] at (-0.45,-0.17) {$t_c$};
    \draw (0.829,0.557) circle (0.3mm);
    \node[scale=1.3] at (1,0.45) {$u_c$};
    \draw (0.213,0.143) circle (0.3mm);
    \node[scale=1.3] at (0.15,0.26) {$v_c$};
    \draw (-5.219,-3.509) -- (5.219,3.509);
    \end{tikzpicture}
    \caption{The intersection points of Theorem \ref{thm_5pointschordal} for $a=0.5$ and $b=0.6e^i$. The circle marked with the black solid outline is the unit circle $S^1$. All other circles and arcs are stereographic projections of great circles.}
    \label{fig_gcp}
\end{figure}

\bigskip

\begin{theorem}\label{thm_5pointschordal}
If $a,b\in\mathbb{B}^2$ are non-collinear with the origin, then the points $k_c,s_c,t_c,u_c,v_c$ in \eqref{equ_kstuv_c} are on a line passing through the origin.
\end{theorem}
\begin{proof}
Since the points $k_c,s_c,t_c,u_c,v_c$ are all given as 
stereographic projections of intersections of two great circles, we can use the formula \eqref{equ_gcis} and its function $F$.

For $t_c$, let
$$
T_c:=F(z;b_{\ast},a^{\ast},b^{\ast},a_{\ast}).
$$

By using the symbolic computation system Risa/Asir, we have
\begin{align*}
 numerator(T_c)=&2|a-b|^3|1-a\overline{b}|
                 (1-|b|^2)(1-|a|^2)(a\overline{b}-\overline{a}b)
                 \big(|a-b|+|1-a\overline{b}|\big) \\
               &\Big(\big(\overline{a}(1-|b|^2)+\overline{b}(1-|a|^2)\big)z^2
                   +2|1-a\overline{b}|(|a-b|-|1-a\overline{b}|)z \\
               &    -\big(a(1-|b|^2)+b(1-|a|^2)\big)\Big), \\
 denominator(T_c)=& |a|^2|b|^2
                   \big|(\overline{a}b-1)|a-b|+b(\overline{b}
                        -\overline{a})|1-a\overline{b}|\big|^2\\
                  & \times
                 \big|(\overline{a}b-1)|a-b|
                        +\overline{a}(a-b)|1-a\overline{b}|\big|^2.
\end{align*}

Because of the assumption that the two points $a,b$ are not collinear with the origin, the inequality $a\overline{b}-\overline{a}b\neq0$ holds.
Therefore, $t_c$ can be obtained by solving the following equation for $z$: 
\begin{equation}\label{eq:tc}
 \big(\overline{a}(1-|b|^2)+\overline{b}(1-|a|^2)\big)z^2
                   +2|1-a\overline{b}|(|a-b|-|1-a\overline{b}|)z 
                -\big(a(1-|b|^2)+b(1-|a|^2)\big)=0.
\end{equation}

\bigskip
For $u_c$, we have
\begin{align*}
&F(z;b,a^{\ast},a,b^{\ast}) \\
& \quad
    =\frac{-2(a\overline{b}-\overline{a}b)
      \Big(\big(\overline{a}(1-|b|^2)+\overline{b}(1-|a|^2)\big)z^2
           -\big(a(1-|b|^2)+b(1-|a|^2)\big)\Big)}
       {|ab|^2}.
\end{align*}
Hence,
$ u_c $ can be obtained by solving the equation
\begin{equation}\label{eq:uc}
 \big(\overline{a}(1-|b|^2)+\overline{b}(1-|a|^2)\big)z^2
           -\big(a(1-|b|^2)+b(1-|a|^2)\big)=0.
\end{equation}

\bigskip

For $s_c$, set
$$
S_c:=F(z;b,a_{\ast},a,b_{\ast}).
$$
By eliminating the non-zero factors from $numerator(S_c)=0$ under the assumptions of the theorem, 
we find that $s_c$ can be obtained as the solution of 
the following equation
\begin{equation}\label{eq:sc}
 \big(\overline{a}(1-|b|^2)+\overline{b}(1-|a|^2)\big)z^2
   +2|1-a\overline{b}|(|1-a\overline{b}|-|a-b|)z
   -\big(a(1-|b|^2)+b(1-|a|^2)\big)=0.
\end{equation}

For $v_c$, set
$$ 
V_c:=F(z;b,b_{\ast},a,a_{\ast}),
$$
and again, by eliminating the non-zero factors from $numerator(V_c)=0$, we have the following equation
\begin{equation}\label{eq:vc}
\begin{aligned}
 &\big(\overline{a}(1-|b|^2)+\overline{b}(1-|a|^2)\big)z^2
   +2\frac{(1-|a|^2)(1-|b|^2)|1-a\overline{b}|}{|1-a\overline{b}|-|a-b|}z\\
   &-\big(a(1-|b|^2)+b(1-|a|^2)\big)=0.
\end{aligned}
\end{equation}
(Note that, if $ |1-a\overline{b}|=|a-b| $, then either $ |a|=1 $ or $ |b|=1 $ holds by \eqref{myahl}.)

For $k_c$, set
$$ 
K_c:=F(z;b^{\ast},b_{\ast},a^{\ast},a_{\ast}).
$$
By eliminating the non-zero factors from $numerator(K_c)=0$, we have the following.
\begin{equation}\label{eq:kc}
\begin{aligned}
 &\big(\overline{a}(1-|b|^2)+\overline{b}(1-|a|^2)\big)z^2
   +2\frac{(1-|a|^2)(1-|b|^2)|1-a\overline{b}|}{|a-b|-|1-a\overline{b}|}z\\
   &-\big(a(1-|b|^2)+b(1-|a|^2)\big)=0.
\end{aligned}
\end{equation}

Let $H=a(1-|b|^2)+b(1-|a|^2)$ as in Table \ref{t1}.
Then the equations \eqref{eq:tc}, \eqref{eq:uc}, \eqref{eq:sc}, \eqref{eq:vc},
and \eqref{eq:kc} can be written in the form
\begin{equation}\label{eq:QR}
\overline{H}z^2+2Rz-H=0,\quad(R\in\R),
\end{equation}
for some real $R$, whose exact value can be seen from Table \ref{tgc}.

Note that the solution $z$ of equation \eqref{eq:QR} has the form
$$ 
 \frac{-R+\sqrt{R^2+\vert H\vert^2}}{\overline{H}}
 =\frac{-R+\sqrt{R^2+\vert H\vert^2}}{\vert H\vert^2}H \qquad \mbox{or}\qquad
 \frac{-R-\sqrt{R^2+\vert H\vert^2}}{\vert H\vert^2}H. 
$$
In both cases, the solution has the form $z=MH$ for some real value $M$. Therefore, ${\rm Arg}(z)={\rm Arg}(H)$ for $M>0$ and ${\rm Arg}(z)={\rm Arg}(H)+\pi$ for $M<0$. Hence, it follows that the five points $ b_c,u_c,s_c,v_c,k_c $ are on the line passing through $H$ and the origin.
\end{proof}

\begin{table}[ht]
    \centering
    \begin{tabular}{|l|l|l|}
         \hline
         \multirow{2}{*}{Point} & \multirow{2}{*}{Definition} & \multirow{2}{*}{$R$}\\
         &&\\
         \hline
         \multirow{2}{*}{$k_c$} & \multirow{2}{*}{${\rm GCIS}[a_{\ast},a^{\ast},b_{\ast},b^{\ast}]$} & \multirow{3}{*}{$\dfrac{(1-|a|^2)(1-|b|^2)|1-a\overline{b}|}{|a-b|-|1-a\overline{b}|}$}\\
         &&\\
         &&\\
         \hline
         \multirow{2}{*}{$s_c$} & \multirow{2}{*}{${\rm GCIS}[a,b_{\ast},b,a_{\ast}]$} &  \multirow{2}{*}{$|1-a\overline{b}|(|1-a\overline{b}|-|a-b|)$}\\
         &&\\
         \hline
         \multirow{2}{*}{$t_c$} & \multirow{2}{*}{${\rm GCIS}[a_{\ast},b^{\ast},b_{\ast},a^{\ast}]$} & \multirow{2}{*}{$|1-a\overline{b}|(|a-b|-|1-a\overline{b}|)$}\\
         &&\\
         \hline
         \multirow{2}{*}{$u_c$} & \multirow{2}{*}{${\rm GCIS}[a,b^{\ast},b,a^{\ast}]$} & \multirow{2}{*}{0}\\
         &&\\
         \hline
         \multirow{2}{*}{$v_c$} & \multirow{2}{*}{${\rm GCIS}[a,a_{\ast},b,b_{\ast}]$} & \multirow{3}{*}{$\dfrac{(1-|a|^2)(1-|b|^2)|1-a\overline{b}|}{|1-a\overline{b}|-|a-b|}$}\\
         &&\\
         &&\\
         \hline
    \end{tabular}
    \medskip
    \caption{The real numbers $R$ with which the formulas of the intersection points $k_c,s_c,t_c,u_c,v_c$ from \eqref{equ_kstuv_c} can be written as the roots of the equation $\overline{H}z^2+2Rz-H=0$, where $H=a(1-|b|^2)+b(1-|a|^2)$.}
    \label{tgc}
\end{table}

\begin{remark}
The equation \eqref{eq:QR} has two solutions in the form $z$ and $-1/\overline{z}$. Assuming $R\neq0$, there is one solution inside and one outside the unit disk. Therefore, under the condition $ a,b\in\mathbb{B}^2$, every equation of $T_c=0$, $S_c=0$, $V_c=0$, and $K_c=0$ in the proof of Theorem \ref{thm_5pointschordal} has a unique solution in the unit disk, and the points $t_c$, $s_c$, $v_c$, and $k_c$ are well defined from \eqref{equ_gcis}. While there are two possible choices for $u_c$, it does not matter which of them is considered as they both are collinear with the other points. Furthermore, the intersection points $k_c,s_c,t_c,u_c,v_c$ of the stereographic projections of great circles in Theorem \ref{thm_5pointschordal} are also collinear with the corresponding intersection points of $k_c,s_c,t_c,v_c$ outside the unit circle. The proof for this claim follows from the fact that these intersection points in the complement of the unit disk also fulfill the condition $F(z;a,b,c,d)=0$ in \eqref{equ_gcis}. Their collinearity can be seen from Figure \ref{fig_gcp}.
\end{remark}

\begin{corollary}\label{cor_11points}
If $a,b\in\B^2$ are non-collinear with the origin, $m$ is the hyperbolic midpoint of $a$ and $b$, the five points $k,s,t,u,v$ are as in \eqref{equ_kstuv} for these choices of $a$ and $b$, and the other five $k_c,s_c,t_c,u_c,v_c$ are as in \eqref{equ_kstuv_c}, then the 11 points
$k,m,s,t,u,v,k_c,s_c,t_c,u_c,v_c$ and the origin are collinear.
\end{corollary}
\begin{proof}
As can be seen from Theorem \ref{thm_formulas} or Table \ref{t1}, the arguments of points $k,m,s,t,u,v$ are all equal to either ${\rm Arg}(H)$ or ${\rm Arg}(H)+\pi$ with $H=a(1-|b|^2)+b(1-|a|^2)$, and
they therefore coincide with the arguments of points $k_c,s_c,t_c,u_c,v_c$
in Theorem \ref{thm_5pointschordal}.
\end{proof}

\bigskip

\begin{figure}
    \centering
    \begin{tikzpicture}[scale=3]
    \draw (0.5,0) circle (0.3mm);
    \node[scale=1.3] at (0.61,0.03) {$a$};
    \draw (0.324,0.504) circle (0.3mm);
    \node[scale=1.3] at (0.23,0.53) {$b$};
    \draw (2,0) circle (0.3mm);
    \node[scale=1.3] at (2.17,0) {$a^*$};
    \draw (0.900,1.402) circle (0.3mm);
    \node[scale=1.3] at (0.93,1.55) {$b^*$};
    \draw (0.943,-0.330) circle (0.3mm);
    \node[scale=1.3] at (1.07,-0.27) {$a_*$};
    \draw (0.401,0.916) circle (0.3mm);
    \node[scale=1.3] at (0.3,1.07) {$b_*$};
    \draw (0,0) circle (0.3mm);
    \node[scale=1.3] at (0,0.13) {$0$};
    \draw (0,0) circle (1cm);
    \draw[dashed] (1.25,0.544) circle (0.926cm);
    \draw (0.5,0) -- (0.401,0.916);
    \draw (2,0) -- (0.324,0.504);
    \draw (0.449,0.467) circle (0.3mm);
    \node[scale=1.3] at (0.41,0.33) {$p$};
    \draw (0.5,0) -- (0.900,1.402);
    \draw (2,0) -- (0.401,0.916);
    \draw (0.710,0.738) circle (0.3mm);
    \node[scale=1.3] at (0.67,0.85) {$q$};
    \draw (0.5,0) arc (345.2-360:27.0:1.292);
    \draw (2,0) arc (41.7:104.7:1.675);
    \draw (0.524,0.544) circle (0.3mm);
    \node[scale=1.3] at (0.53,0.69) {$p_c$};
    \draw (0.5,0) arc (318.2-360:9.8:1.675);
    \draw (2,0) arc (14.7:105.6:1.292);
    \draw (0.917,0.953) circle (0.3mm);
    \node[scale=1.3] at (1,0.83) {$q_c$};
    \draw (1.9,1.973) -- (-1,-1.038);
    \end{tikzpicture}
    \caption{The intersection points of Theorem \ref{thm_pqc} for $a=0.5$ and $b=0.6e^i$. The circle with solid line is the unit circle and the dashed circle is the orthogonal circle to the unit circle passing through $a,b$. All the circle arcs are arcs from the stereographic projections of great circles passing through the end points of the arc in question.}
    \label{fig_pqc}
\end{figure}
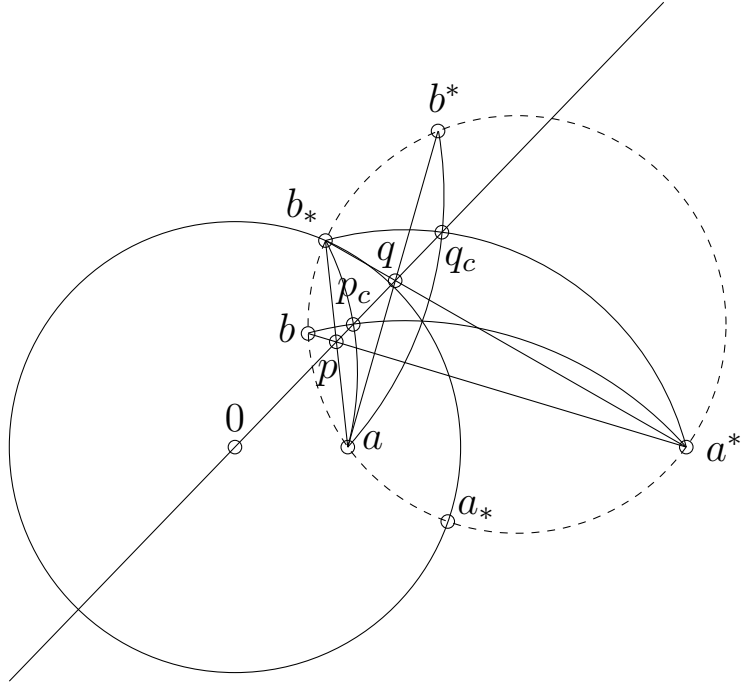

\begin{theorem}\label{thm_pqc}
Choose such $a,b\in\B^2\backslash\{0\}$ non-collinear with the origin and let $a^*,b^*,a_*,b_*$ be as in in \eqref{eq:ab-ast} and \eqref{eq:ab-end}. Fix $p={\rm LIS}[a,b_*,a^*,b]$ and $q={\rm LIS}[a,b^*,a^*,b_*]$. Similarly, let $p_c={\rm GCIS}[a,b_*,a^*,b]$, and $q_c={\rm GCIS}[a,b^*,a^*,b_*]$. The intersection points $p,q,p_c,q_c$ and the origin are collinear.
\end{theorem}
\begin{proof}
For $q_c$, calculating  $F(z;a,b^{\ast},a^{\ast},b_{\ast})=0$,
eliminating the non-desired factors, 
$ q_c $ is obtained by the solution of the following equation,
\begin{equation}\label{eq:qc}
\begin{aligned}
&Qz^2+Rz-\overline{Q}=0,\quad \text{where}\\
&Q=\overline{b}(1-|a|^2)^2+\overline{a}|a-b|(|1-\overline{a}b|-|a-b|),\\
&R=(1-|a|^2)|1-\overline{a}b|(|a-b|-|1-\overline{a}b|).
\end{aligned}    
\end{equation}
Note that $R\in\R$ above.

\bigskip

Similarly for $p_c$, the point $p_c$ is given by the equation $Qz^2+Rz+\overline{Q}=0$, where $Q$ and $R$ are as in \eqref{eq:qc}, and the arguments of the solutions $q_c$ and $p_c$ are therefore ${\rm Arg}(\overline{Q})$ or ${\rm Arg}(\overline{Q})+\pi$.

\bigskip

For points $p$ and $q$, the following is also obtained from 
calculation of ${\rm LIS}[a,b^{\ast},a^{\ast},b_{\ast}]$ and 
${\rm LIS}[a,b_{\ast},a^{\ast},b]$, respectively.
\begin{align*}
 q&=\frac{b(1-|a|^2)^2+a|a-b|(|1-\overline{a}b|-|a-b|)}
     {|b|^2(1-|a|^2)^2+|a-b|(|1-\overline{a}b|-|a-b|)}, \\
 p&=\frac{b(1-|a|^2)^2+a|a-b|(|1-\overline{a}b|-|a-b|)}
    {(1-|a|^2)^2+|a|^2|a-b|(|1-\overline{a}b|-|a-b|)}.
\end{align*}
Since the denominators of $p$ and $q$ both have positive real values, clearly,
$$
  {\rm Arg}(q)={\rm Arg}(p)={\rm Arg}\big(b(1-|a|^2)^2+a|a-b|(|1-\overline{a}b|-|a-b|)\big)={\rm Arg}(\overline{Q})
$$
hold. Hence, the four points $ q_c,p_c,q,p $ are on the line passing through the origin and $\overline{Q}=b(1-|a|^2)^2+a|a-b|(|1-\overline{a}b|-|a-b|)$.
\end{proof}

\section{Hyperbolic and chordal midpoints}

In this section, we first introduce a few results related to the hyperbolic midpoint. In Theorems \ref{thm_hmpc} and \ref{thm_hmphinv}, we show new ways to construct the hyperbolic midpoint of two points in the unit disk. This geometric construction reveals the simple geometry not apparent
from the explicit formula in Theorem \ref{myhmidp} for the hyperbolic midpoint. Similar constructions for the hyperbolic midpoint can be found in the book \cite{a16} by Ahara. At the end of this section, Theorem \ref{thm_cmp} offers an explicit formula for the chordal midpoint of two points in the unit disk.

\begin{lemma}\label{lem_abk}
If $V\subset\B^2$ is a lens-shaped domain symmetric with respect
to the real axis bounded by two circular arcs with endpoints $-1$ and $1$, and $a,b\in\partial V$ so that ${\rm Im}\,b<0<{\rm Im}\,a$, then the hyperbolic midpoint of $a$ and $b$ with respect to the domain $\B^2$ is on the real axis.
\end{lemma}
\begin{proof}
Let $k$ be the intersection point of the hyperbolic line $J^*[a,b]$ and the real axis. Trivially, if $k=0$, then $J^*[a,b]$ is a diameter of the unit disk $\B^2$ and, by the symmetry of the domain $V$, $a=-b$ and the hyperbolic midpoint is the origin. Let us next assume that $k\neq0$. Now, the line $J^*[a,b]$ is the arc of the circle that contains $a,b$ and is orthogonal to the unit circle $S^1$. Consider the M\"obius transformation $T_k:\overline{\C}\to\overline{\C}$, defined as in \eqref{myT}. Clearly, $T_k(k)=0$, $T_k(1)=1$, and $T_k(-1)=-1$, and it also follows that the mapping preserves the unit disk $\B^2$, the lens-shaped domain $V$, and its boundary $\partial V$. As $T_k(k)=0$ for a point $k\in J^*[a,b]$, the mapped hyperbolic line is a diameter of $\B^2$. Since the points $T_k(a)$ and $T_k(b)$ are the intersection points of this diameter and the domain $V$, it is clear that $T_k(a)=-T_k(b)$ and $T_k(k)=0$ is their hyperbolic midpoint. Thus, $k$ is the hyperbolic midpoint of $a$ and $b$, and it is on the real axis.   
\end{proof}

\begin{figure}
    \centering
    \begin{tikzpicture}[scale=3]
    \draw (-0.480,0.306) circle (0.3mm);
    \node[scale=1.3] at (-0.45,0.43) {$a$};
    \draw (-0.105,-0.382) circle (0.3mm);
    \node[scale=1.3] at (-0.13,-0.5) {$b$};
    \draw (-0.268,0) circle (0.3mm);
    \node[scale=1.3] at (-0.3,-0.13) {$k$};
    \draw (-0.207,0.371) circle (0.3mm);
    \node[scale=1.3] at (-0.03,0.5) {$T_k(a)$};
    \draw (0.207,-0.371) circle (0.3mm);
    \node[scale=1.3] at (0.5,-0.5) {$T_k(b)$};
    \draw (0,0) circle (1cm);
    \node[scale=1.3] at (0.07,0.1) {$0$};
    \draw (-0.018,-0.999) arc (358.9-360:52.6:1.976);
    \draw (1,0) arc (47.7:132.2:1.486);
    \draw (-1,0) arc (227.7:312.2:1.486);
    \draw (-0.487,0.872) -- (0.487,-0.872);
    \draw (0,1.3) -- (0,-1.3);
    \draw (0.04,1.25) -- (0,1.3) -- (-0.04,1.25);
    \draw (1.3,0) -- (-1.3,0);
    \draw (1.25,0.04) -- (1.3,0) -- (1.25,-0.04);
    \draw (0,0) circle (0.3mm);
    \end{tikzpicture}
    \caption{The hyperbolic line $J^*[a,b]$ before and after the M\"obius transformation $T_k$ for $a=-1.1i+\sqrt{2.21}e^{1.9i}$, $b=1.1i-\sqrt{2.21}e^{1.5i}$, and $k=J^*[a,b]\cap\R$, and the lens-shaped domain $V=B^2(-1.1i,\sqrt{2.21})\cap B^2(1.1i,\sqrt{2.21})$.}
    \label{fig_kinv}
\end{figure}
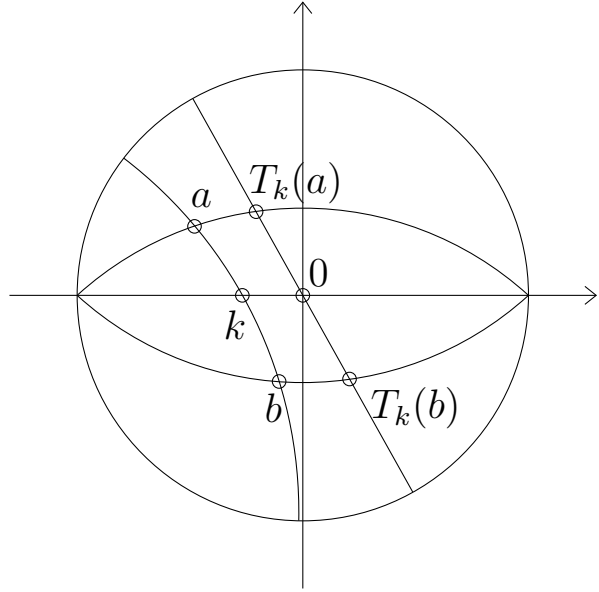

\begin{theorem}\label{thm_hmpc}
Fix $a,b\in\B^2\backslash\{0\}$ so that $a,b$ are not collinear with the origin. Let $c$ be the center of the great circle through $a$ and $b/|b|^2$, and fix $u$ so that $\{u,-u\}=S^1\cap S^1(c,|a-c|)$. The hyperbolic midpoint $m$ of $a$ and $b$ is the intersection point $[-u,u]\cap S^1(v,|a-v|)$, where $v$ is the center of the circle through $a$, $b$, and $a/|a|^2$.
\end{theorem}
\begin{proof}
The stereographic projections of two great circles passing through $a,b/|b|^2$ and $b,a/|a|^2$ are symmetric with respect to the line $L(u,-u)$, so the points $a$ and $b$ are on different sides of $L(u,-u)$ on the boundary of a lens-shaped domain $B^2(c,|c-u|)\cap B^2(c,|c-u|)$. The circle $S^1(v,|a-v|)$ passing through $a$, $b$, and $a/|a|^2$ is orthogonal to the unit circle so $J^*[a,b]=S^1(v,|a-v|)\cap\B^2$. Consequently, the hyperbolic midpoint $m$ must be the intersection point $J^*[a,b]\cap[u,-u]$ and, by rotating all the points around the origin until $u$ is at 1, we see that this result follows from Lemma \ref{lem_abk}.
\end{proof}

\begin{remark}
The fact that $\{u,-u\}=S^1\cap S^1(c,|a-c|)$ for a complex point $u$ in Theorem \ref{thm_hmpc} follows from \cite[Ex. 3.15(1), p. 33]{hkv}.
\end{remark}

\begin{theorem}\label{thm_hmphinv}
Fix $a,b\in\B^2\backslash\{0\}$ so that $a,b$ are not collinear with the origin and $|a|\neq|b|$. Let $c={\rm LIS}[a,b,a_*,b_*]$, where $a_*,b_*$ are as in \eqref{eq:ab-end}. Then the point $m=J^*[a,b]\cap S^1(c,\sqrt{|c|^2-1})$ is the hyperbolic midpoint of $a$ and $b$.
\end{theorem}
\begin{proof}
Clearly, the circle $S^1(c,\sqrt{|c|^2-1})$ is orthogonal to the unit circle $S^1$ and it passes through $m$. Let $h$ be the inversion in the circle $S^1(c,\sqrt{|c|^2-1})$. Now, $h(m)=m$, $h(a)=b$, $h(a_*)=b_*$ as this inversion preserves the unit disk $\B^2$ and the angle magnitudes. It follows that
\begin{align*}
|b_*,b,m,a_*|=|h(a_*),h(a),h(m),h(b_*)|=|a_*,a,m,b_*|.  
\end{align*}
Thus, $m$ is the hyperbolic midpoint of $a$ and $b$ by \cite[(4.9)]{hkv} or by \eqref{epdefrho}. 
\end{proof}

\begin{figure}
    \centering
    \begin{tikzpicture}[scale=3]
    \draw (0.412,0.282) circle (0.3mm);
    \node[scale=1.3] at (0.3,0.3) {$a$};
    \draw (0.672,-0.195) circle (0.3mm);
    \node[scale=1.3] at (0.63,-0.29) {$b$};
    \draw (0.525,0.851) circle (0.3mm);
    \node[scale=1.3] at (0.6,0.95) {$a_*$};
    \draw (0.934,-0.356) circle (0.3mm);
    \node[scale=1.3] at (1.07,-0.35) {$b_*$};
    \draw (1.223,-1.211) circle (0.3mm);
    \node[scale=1.3] at (1.25,-1.1) {$c$};
    \draw (0.515,-0.001) circle (0.3mm);
    \node[scale=1.3] at (0.37,0.01) {$m$};
    \draw (1.223,-1.211) -- (0.412,0.282);
    \draw (1.223,-1.211) -- (0.525,0.851);
    \draw (0,0) circle (1cm);
    \node[scale=1.3] at (0,0.15) {$0$};
    \draw (0,0) circle (0.3mm);
    \draw (1.223,0.190) arc (90:180:1.401);
    \draw (0.525,0.851) arc (148.3:249.0:0.827);
    \end{tikzpicture}
    \caption{The hyperbolic midpoint $m$ of $a=0.5e^{0.6i}$ and $b=0.7e^{6i}$ in the unit disk, the hyperbolic line $J^*[a,b]$ with end points $a_*$ and $b_*$, and an arc of the circle $S^1(c,\sqrt{|c|^2-1})$, where $c$ is the intersection point of Euclidean lines $L(a,b)$ and $L(a_*,b_*)$. }
    \label{fig_hmphInv}
\end{figure}
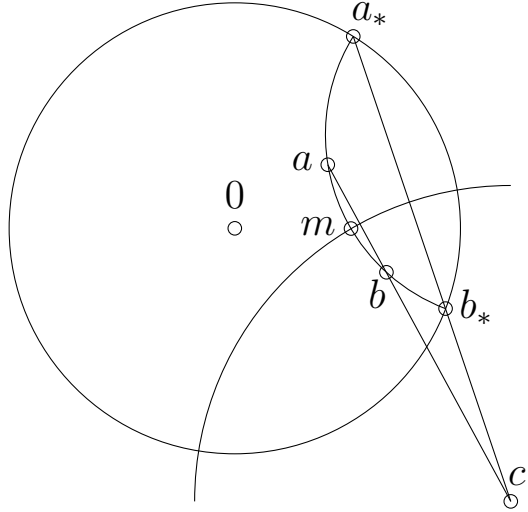

\begin{remark}
In Theorem \ref{thm_hmphinv}, the hyperbolic line $J^*[a,b]$ is orthogonal to the circle $S^1(c,\sqrt{|c|^2-1})$ because of the choice of $c$.
\end{remark}

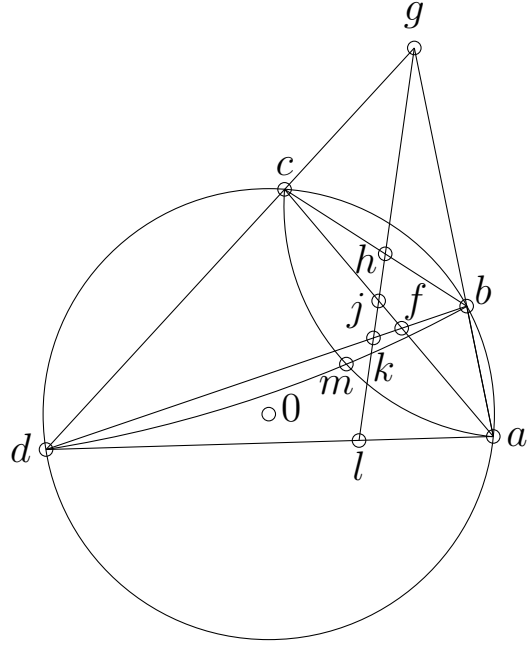
\begin{figure}[ht] %
    \begin{tikzpicture}[scale=3]
    \draw (0.995,-0.099) circle (0.3mm);
    \node[scale=1.3] at (1.1,-0.1) {$a$};
    \draw (0.877,0.479) circle (0.3mm);
    \node[scale=1.3] at (0.95,0.55) {$b$};
    \draw (0.070,0.997) circle (0.3mm);
    \node[scale=1.3] at (0.070,1.1) {$c$};
    \draw (-0.987,-0.157) circle (0.3mm);
    \node[scale=1.3] at (-1.1,-0.15) {$d$};
    \draw (0.645,1.624) circle (0.3mm);
    \node[scale=1.3] at (0.65,0.50) {$f$};
    \draw (0.344,0.222) circle (0.3mm);
    \node[scale=1.3] at (0.645,1.77) {$g$};
    \draw (0.589,0.381) circle (0.3mm);
    \node[scale=1.3] at (0.3,0.13) {$m$};
    \draw (0.516,0.711) circle (0.3mm);
    \node[scale=1.3] at (0.43,0.69) {$h$};
    \draw (0.487,0.502) circle (0.3mm);
    \node[scale=1.3] at (0.39,0.46) {$j$};
    \draw (0.464,0.338) circle (0.3mm);
    \node[scale=1.3] at (0.51,0.21) {$k$};
    \draw (0.400,-0.117) circle (0.3mm);
    \node[scale=1.3] at (0.4,-0.23) {$l$};
    \draw (0.645,1.624) -- (0.995,-0.099);
    \draw (0.645,1.624) -- (-0.987,-0.157);
    \draw (0.645,1.624) -- (0.400,-0.117);
    \draw (0.995,-0.099) -- (0.877,0.479);
    \draw (0.995,-0.099) -- (0.070,0.997);
    \draw (0.995,-0.099) --  (-0.987,-0.157);
    \draw (0.877,0.479) -- (0.070,0.997);
    \draw (0.877,0.479) -- (-0.987,-0.157);
    \draw (0,0) circle (1cm);
    \draw (0,0) circle (0.3mm);
    \node[scale=1.3] at (0.1,0.03) {$0$};
    \draw (0.070,0.997) arc (175.9:264.2:1.029);
    \draw (-0.987,-0.157) arc (279.0:298.6:5.797);
    \end{tikzpicture}
    \caption{The intersection points of Lemma \ref{lem_0fm}, Corollary \ref{cor_mhypof}, and Conjecture \ref{con_rhofgh} for $a=e^{-0.1i}$, $b=e^{0.5i}$, $c=e^{1.5i}$, $d=e^{3.3i}$, and $h=b+0.447(c-b)$.}
    \label{fig:hgf}
\end{figure}

\begin{thm}\label{myoarcisec} \cite{wvz}
Let  $a,b,c,d \in S^1$ be points listed in the order they occur when one traverses the unit circle in the positive direction.

(1) Let
$$w_1={\rm LIS}[a,b,c,d],\quad w_2={\rm LIS}[a,c,b,d],\quad w_3= {\rm LIS}[a,d,b,c]\,.$$
Then the point $w_2$ is the orthocenter of the triangle with vertices at the points $0$, $w_1$, $w_3.$

(2)
The point of intersection of the hyperbolic lines $J^*[a,c]$ and $J^*[b,d]$ is given by
\begin{equation*}\label{myoaisec}
w= \frac{(ac-bd)\pm\sqrt{(a-b)(b-c)(c-d)(d-a)}}{a-b+c-d}\,,
\end{equation*}
where the sign "$+$" or "$-$" in front of the square
root is chosen such that $|w|<1\,.$
 \end{thm}

\begin{lemma}\label{lem_0fm}
Fix four complex points $a,b,c,d$ on the unit circle $S^1$ in this order. Let $f$ be the intersection point of the Euclidean lines $L(a,c)$ and $L(b,d)$, and $m$ the intersection point of the hyperbolic lines $J^*[a,c]$ and $J^*[b,d]$. If $f\neq0$, then $f$ and $m$ are collinear with the origin.
\end{lemma}
\begin{proof}
The points $m$ and $f$ are collinear with the origin, if and only if $m\slash f\in\R\backslash\{0\}$ or, equivalently, $m\overline{f}\in\R\backslash\{0\}$. Here, $m=0$ or $\overline{f}=0$ is possible only in the special case $f=m=0$ and, since $f\neq0$, we only need to show $m\overline{f}\in\R$. Theorem \ref{myoarcisec}(2) gives an explicit formula for $m$ and the complex conjugate of $f$ is as in Proposition \ref{prop_conjIntF}. By using these formulas, we will have
\begin{align}\label{quoRoot}
m\overline{f}
=1\pm\frac{\sqrt{(a-b)(b-c)(c-d)(d-a)}}{ac-bd}.
\end{align}
The expression under the square root can be written as
\begin{align*}
&(a-b)(b-c)(c-d)(d-a)
=(ac+bd-ad-bc)(ac+bd-ab-cd)\\
&=(ac+bd)^2-(ac+bd)(ab+cd)-(ac+bd)(ad+bc)+(ab+cd)(ad+bc).
\end{align*}
Hence, the quotient in \eqref{quoRoot} is 
\begin{align*}
\sqrt{\left(\frac{ac+bd}{ac-bd}\right)^2-\frac{(ac-bd)(ab+cd)}{(ac-bd)^2}-\frac{(ac-bd)(ad+bc)}{(ac-bd)^2}+\frac{(ab+cd)(ad+bc)}{(ac-bd)^2}}
\end{align*}
Because $z\overline{z}=|z|^2=1$ for all $z\in S^1$ and $y-\overline{y}\in i\R$ for all $y\in i\R$, we will have
\begin{align*}
\frac{ac+bd}{ac-bd}
&=\frac{(ac+bd)\overline{(ac-bd)}}{|ac-bd|^2}
=\frac{\overline{ac}bd-ac\overline{bd}}{|ac-bd|^2}
\in i\R,\\
\frac{ab+cd}{ac-bd}
&=\frac{(ab+cd)\overline{(ac-bd)}}{|ac-bd|^2}
=\frac{\overline{a}d-a\overline{d}+b\overline{c}-b\overline{c}}{|ac-bd|^2}\in i\R,\\
\frac{ad+bc}{ac-bd}
&=\frac{(ad+bc)\overline{(ac-bd)}}{|ac-bd|^2}
=\frac{\overline{a}b-a\overline{b}+\overline{c}d-c\overline{d}}{|ac-bd|^2}\in i\R.
\end{align*}
Consequently, it follows that the quotient in \eqref{quoRoot} consists of such products of imaginary numbers that are all real numbers and therefore $m\overline{f}$ is also a real number.
\end{proof}

\begin{corollary}\label{cor_mhypof}
If $a,b,c,d$ are four complex points on the unit circle $S^1$ in this order so that the intersection point $f$ of the Euclidean lines $L(a,c)$ and $L(b,d)$ is not the origin, then the intersection point $m$ of the hyperbolic lines $J^*[a,c]$ and $J^*[b,d]$ is the hyperbolic midpoint of the origin and the point $f$.
\end{corollary}
\begin{proof}
Since $f\neq0$, the origin cannot be on both of the lines $L(a,c)$ and $L(b,d)$, which means that at least one of the points pairs $a,c$ and $b,d$ are non-collinear with the origin. By symmetry, we can suppose that $a,c$ are non-collinear with the origin. By Lemma \ref{lem_0fm}, the points $f$ and $m$ are collinear with the origin and the inequality $|m|<|f|$ must hold because $J^*[a,c]$ is an arc of a circle orthogonal to the unit circle $S^1$. Consequently, $m=[0,f]\cap J^*[a,c]$, where the point $f$ trivially fulfills $f\in[a,c]\cap\B^2$. The corollary follows from \cite[Prop. 3.1(1), p. 447]{vw}.
\end{proof}

\begin{conjecture}\label{con_rhofgh}
Let $a,b,c,d$ be four complex points on the unit circle $S^1$ in this order so that $L(a,b)$ and $L(c,d)$ are not parallel. Let $h$ be an arbitrary point on the Euclidean segment $[b,c]$, and fix then $g={\rm LIS}[a,b,c,d]$, $j={\rm LIS}[g,h,a,c]$, $k={\rm LIS}[g,h,b,d]$, and $l={\rm LIS}[g,h,a,d]$. Note that the special case $j=k$ is possible. Now, $\rho_{\mathbb{B}^2}(h,j)=\rho_{\mathbb{B}^2}(k,l)$. See Figure \ref{fig:hgf}.   
\end{conjecture}

\begin{remark}
To prove Conjecture \ref{con_rhofgh}, it might be useful to consider the M\"obius transformation $T_{g^*}$ from \eqref{myT}, where $g^*=1/\overline{g}$. It preserves the unit disk and circle, but it maps $g$ to the infinity. Numerical tests suggest that $|T_{g^*}(h)|=|T_{g^*}(l)|$ and $|T_{g^*}(j)|=|T_{g^*}(k)|$ and, since the mapped points $T_{g^*}(h),T_{g^*}(j),T_{g^*}(k),T_{g^*}(l)$ are collinear, this would be enough to prove that $\rho_{\mathbb{B}^2}(T_{g^*}(h),T_{g^*}(j))=\rho_{\mathbb{B}^2}(T_{g^*}(k),T_{g^*}(l))$, from which the result follows because of the conformal invariance of the hyperbolic metric.
\end{remark}

\bigskip

\begin{figure}
    \centering
    \includegraphics[scale=0.6]{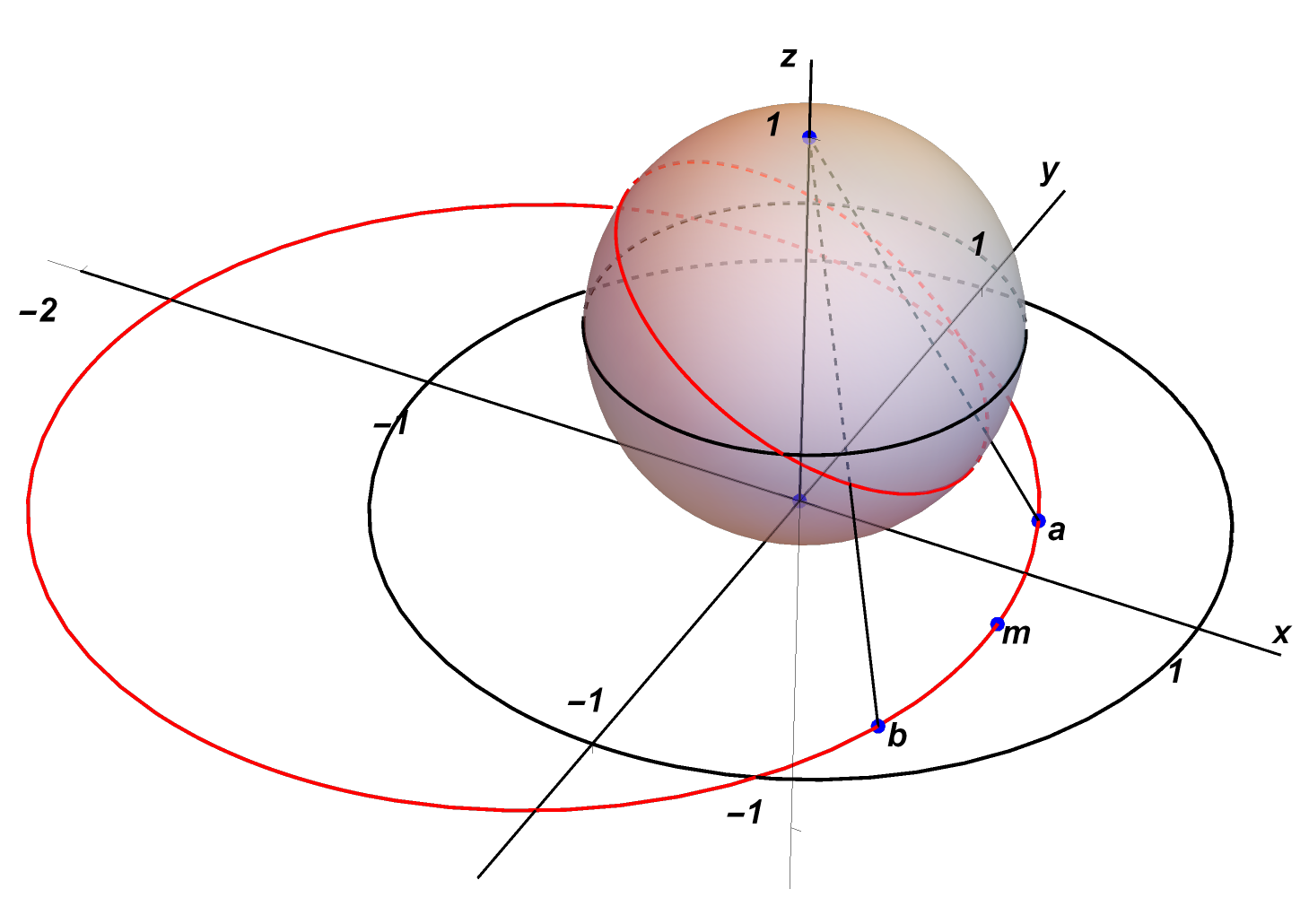}
    \caption{Two points $a,b$ in the unit disk, their chordal midpoint $m$, and the stereographic projection of the great circle passing through $a,b,m$ on the complex plane with the Riemann sphere. This figure is a courtesy of Heikki Ruskeep\"a\"a \cite{r09}.}
    \label{fig_cmp}
\end{figure}

\bigskip

\begin{theorem}\label{thm_cmp}
Let points $ a $ and $ b $ be points on $\mathbb{C} $ 
and assume that the corresponding $ \pi(a)$ and $ \pi(b) $ 
are not symmetric with respect to the center $ (0,0,\frac12) $.
For $ a,b $, the chordal midpoint $ m $ is given by
$$
  m=\frac{a(1+|b|^2)+b(1+|a|^2)}
         {|1+a\overline{b}|\sqrt{(1+|a|^2)(1+|b|^2)}-|ab|^2+1}.
$$
Moreover, the stereographic projection of the great circle passing through 
$ \pi(a) $ and $ \pi(b) $ is given by
$$
   \bigg| z-\frac{a(1-|b|^2)-b(1-|a|^2)}{b\overline{a}-a\overline{b}} \bigg|
      =\frac{|a-b||1+a\overline{b}|}{|a\overline{b}-\overline{a}b|}.
$$
\end{theorem}
\begin{proof}
Let $ a,b\in\mathbb{C} $.
Recall from \eqref{riePoint} that the points on the Riemann sphere $ \pi(a) , \pi(b) $
corresponding to the two points 
$ a $ and $ b $ are respectively given by
$$ 
  \pi(a)=\Big(\frac{\mbox{Re}\,a}{1+|a|^2},\frac{\mbox{Im}\,a}{1+|a|^2},
            \frac{|a|^2}{1+|a|^2}\Big), \quad \mbox{and}\quad
  \pi(b)=\Big(\frac{\mbox{Re}\,b}{1+|b|^2},\frac{\mbox{Im}\,b}{1+|b|^2},
            \frac{|b|^2}{1+|b|^2}\Big). 
$$
The chordal midpoint $ m $ is given by the projection of
the midpoint $ \pi(m)$ on the $ \mbox{arc}(\pi(a), \pi(b)) $ 
of the circle with center $ (0,0,\frac12) $ which passes through 
$ \pi(a) $ and $ \pi(b) $, where the $ \mbox{arc}(\pi(a),\pi(b)) $
is chosen to be a smaller of the two.
The point $ \pi(m) $ is the point on the line passing through 
the Euclidean midpoint $ (\pi(a)+\pi(b))/2 $ and $ (0,0,\frac12) $ 
whose distance from $ (0,0,\frac12) $ is $ 1/2 $.
Therefore, $ \pi(m) $ is given as follows:
\begin{align*}
  &\Big( \frac14\frac{\mbox{Re}\,a(1+|b|^2)+\mbox{Re}\,b(1+|a|^2)}
               {(1+|a|^2)(1+|b|^2)L},\,
        \frac14\frac{\mbox{Im}\,a(1+|b|^2)+\mbox{Im}\,b(1+|a|^2)}
               {(1+|a|^2)(1+|b|^2)L},\\
         &\frac14\frac{1+|ab|^2}{(1+|a|^2)(1+|b|^2)L}+\frac12\Big),    
\end{align*}
where $ L^2=\dfrac14\dfrac{(1+a\overline{b})(1+\overline{a}b)}
                        {(1+|a|^2)(1+|b|^2)} $
and $ L $ is chosen so that $ \pi(m) $ is the point on a 
smaller $ \mbox{arc}(\pi(a),\pi(b)) $.
So, we can assume $L>0$ since such $ \pi(m) $ is in the same direction as the Euclidean midpoint $ (\pi(a)+\pi(b))/2 $ from the center
$ (0,0,\frac12) $.

The point on the complex plane corresponding to the above point is the chordal midpoint $ m $ given by
$$
  m=\frac{a(1+|b|^2)+b(1+|a|^2)}
         {2(1+|a|^2)(1+|b|^2)L-|ab|^2+1}
   =\frac{a(1+|b|^2)+b(1+|a|^2)}
         {|1+a\overline{b}|\sqrt{(1+|a|^2)(1+|b|^2)}-|ab|^2+1}.
$$

To prove the second part of the theorem, note that the equation of the plane $ \{PL=0\}\subset \mathbb{R}^3 $ passing through 
$ \pi(a)$, $\pi(b)$, and $ (0,0,\frac12) $ is given by
\begin{align*}
  PL=& \big(-\mbox{Im}\,{a}(1-|b|^2)+\mbox{Im}\,{b}(1-|a|^2)\big)\xi
   +\big(\mbox{Re}\,{a}(1-|b|^2)-\mbox{Re}\,{b}(1-|a|^2)\big)\eta \\
  &\quad
    +(\mbox{Im}\,{b}\cdot\mbox{Re}\,a-\mbox{Re}\,{b}\cdot\mbox{Im}\,{a})
      (2\zeta-1)=0.
\end{align*}
The great circle with center $ (0,0,\frac12) $ which passes
through $ \pi(a) $ and $ \pi(b) $ is given as the intersection of 
the plane $ PL=0 $ and the sphere $ \xi^2+\eta^2+(\zeta-\frac12)^2=1/4 $.

Consider the system of equations
\begin{equation}\label{eq:sys}
   PL=0,\quad  \xi^2+\eta^2+(\zeta-\frac12)^2=1/4, \quad
   \mbox{Re}\,z=\frac{\xi}{1-\zeta}, \mbox{\quad and \quad}
   \mbox{Im}\,z=\frac{\eta}{1-\zeta}.
\end{equation}
By eliminating $ \xi,\eta,\zeta $ from the above equations, we have
\begin{equation}\label{eq:g-circle}
 (-a\overline{b}+\overline{a}b)z\overline{z}
  +(a\overline{a}\overline{b}-\overline{a}b\overline{b}+\overline{a}
    -\overline{b})z+\big(\big((-\overline{a}+\overline{b})b-1\big)a+b\big)
   \overline{z}+a\overline{b}-\overline{a}b=0.
\end{equation}
The equation \eqref{eq:g-circle} is obtained from the calculation of the Gr\"obner basis by Risa/Asir the symbolic computation system.
This equation can be written as
$$
   \bigg| z-\frac{a(1-|b|^2)-b(1-|a|^2)}{b\overline{a}-a\overline{b}} \bigg|
      =\frac{|a-b||1+a\overline{b}|}{|a\overline{b}-\overline{a}b|},
$$
which gives the equation of the stereographic projection of the intersection $ PL=0 $ and the sphere.

Let us yet prove the last part of the theorem. The equation of the plane $ \{PO=0\}\subset \mathbb{R}^3 $ 
orthogonal to the line passing through
$ \pi(a) $ and $ \pi(b) $ that passes through $ \pi(m) $ is
\begin{align*}
  PO & =2\big(\mbox{Re}\, b(1+|a|^2)-\mbox{Re}\, a(1+|b|^2)\big)\xi
      +2\big(\mbox{Im}\, b(1+|a|^2)-\mbox{Im}\,a(1+|b|^2)\big)\eta\\
     &\qquad
         +2(-|a|^2+|b|^2)\zeta+|a|^2-|b|^2=0.
\end{align*}
Consider the system of equations
\begin{equation}\label{eq:sys2}
   PO=0,\quad  \xi^2+\eta^2+(\zeta-\frac12)^2=1/4, \quad
   \mbox{Re}\,z=\frac{\xi}{1-\zeta}, \mbox{\quad and \quad}
   \mbox{Im}\,z=\frac{\eta}{1-\zeta}.
\end{equation}
By eliminating $ \xi,\eta,\zeta $ from the above equations, we have
$$
   (-|a|^2+|b|^2)z\overline{z}+(\overline{b}(1+|a|^2)-\overline{a}(1+|b|^2))z
             +(b(1+|a|^2)-a(1+|b|^2))\overline{z}+|a|^2-|b|^2=0.
$$
This equation can be written as
$$ 
   \bigg|z-\frac{b(1+|a|^2)-a(1+|b|^2)}{|a|^2-|b|^2}\bigg|
   =\frac{|a-b|\sqrt{(1+|a|^2)(1+|b|^2)}}{\big||a|^2-|b|^2\big|},
$$
and it gives the equation of the stereographic projection of the intersection $ PO=0 $ and the sphere.
\end{proof}

\newpage 
\section{Appendix: Computer algebra and Gr\"obner bases with Risa/Asir}

The proofs in Sections 3-5 require finding the points of intersection of circles and lines in a form as simple as possible. Hence we have a system of polynomial equations which we want to solve exactly. Manual computations are often very difficult due to complicated formulas. Therefore we use the Gr\"obner base method \cite{cls} from computer algebra as a tool and combine it with manual processing. We will here outline our procedure.

\begin{nonsec}{\bf Gr\"obner bases and Risa/Asir} \label{grb} \end{nonsec}
Our procedure using the Gr\"obner bases has usually the following steps
\begin{enumerate}
\item[(1)] Setting up the polynomial equations $f_1=f_2=\cdots=f_n=0$ and defining the ideal $I=\left< f_1,f_2,\cdots,f_n\right>$
\item[(2)] Set variable ordering
\item[(3)] Choice of the monomial ordering (The lexicographic ordering is the most convenient for solving polynomial equations)
\item[(4)] Finding the Gr\"obner basis $G$ for the ideal $I$
\item[(5)] Choose the elimination ideal from $G$
\item[(6)] Manual postprocessing
\end{enumerate}

In this appendix, we show how we used the symbolic computation system Risa/Asir \cite{risa} to find equation 
\eqref{eq:g-circle}. Similar calculations could be performed also by using Mathematica, Maple, or some other system.

Let {\tt z.re}, {\tt z.im}, and {\tt zb} denote the real, 
imaginary, and complex conjugate of {\tt z} respectively.
Then {\tt z = z.re + i*z.im} represents a point on $\mathbb{C}$,
Setting $ \xi=${\tt xx}, $ \eta=${\tt yy}, $ \zeta=${\tt zz},
the system of equation \eqref{eq:sys}
can be expressed in Risa/Asir as follows. Below, "ord" is a command that specifies the order of variables in the output polynomial and "nm" is a command that returns a numerator.

\medskip

\noindent
\newbox\SomeBox
\begin{lrbox}{\SomeBox}
  \begin{minipage}{0.95\textwidth}
\begin{verbatim}
[2062] ord([z,zb,z.re,z.im])$
[2063] PL=(b.im*(1-a*ab)-a.im*(1-b*bb))*xx
        +(a.re*(1-b*bb)-b.re*(1-a*ab))*yy
        +(2*zz-1)*(b.im*a.re-b.re*a.im)$
[2064] RS=xx^2+yy^2+(zz-1/2)^2-1/4$
[2065] Z.xx=nm(xx/(1-zz)-z.re)$
[2066] Z.yy=nm(yy/(1-zz)-z.im)$
\end{verbatim}
  \end{minipage}
\end{lrbox}
\fbox{\usebox{\SomeBox}}

\medskip

Above, {\tt [20**]} are input prompts.

To eliminate {\tt xx}, {\tt yy}, and {\tt zz} from the system of equation, we use the following method.

Let {\tt I} be the ideal generated by polynomials {\tt PL}, {\tt RS}, {\tt Z.xx}, {\tt Z.yy}. We can compute the Gr\"obner basis for the ideal {\tt I} by using the following block order
$$\mbox{\tt  [xx,yy,zz] > [a.re,a.im,b.re,b.im, a,ab,b,bb,z.re,z.im]}. $$
by

\medskip

\noindent
\newbox\SomeBox
\begin{lrbox}{\SomeBox}
  \begin{minipage}{0.95\textwidth}
\begin{verbatim}
[2067] G=nd_gr([PL,RS,Z.xx,Z.yy],
        [xx,yy,zz, a.re,a.im,b.re,b.im, a,ab,b,bb,z.re,z.im],
        0,[[0,3],[0,10]])$
\end{verbatim}
  \end{minipage}
\end{lrbox}
\fbox{\usebox{\SomeBox}}

\medskip

Then, we need to choose the polynomials that belong to the polynomial ring
$$ \mathbb{C}[\mbox{\tt a.re,a.im,b.re,b.im, a,ab,b,bb,z.re,z.im}] $$
from the elements of the Gr\"obner basis {\tt G}.

\medskip

\noindent
\newbox\SomeBox
\begin{lrbox}{\SomeBox}
  \begin{minipage}{0.95\textwidth}
\begin{verbatim}
[2068] length(G);
12
[2069] vars(G[0]);
[a,b,b.im,ab,a.im,bb,a.re,b.re,z.re,z.im]
[2070] vars(G[1]);
[b.im,a.im,a.re,b.re,zz,z.re,z.im]
\end{verbatim}
  \end{minipage}
\end{lrbox}
\fbox{\usebox{\SomeBox}}

\medskip

From the above calculation, we can see that the Gr\"obner basis {\tt G}
is generated by 12 polynomials and that the first polynomial {\tt G[0]}
in the list of elements is the required one. 
(The second {\tt G[1]} and subsequent polynomials in the list are polynomials 
containing variables to be eliminated such as {\tt zz}.)
We can obtain the desired elimination ideal {\tt <G[0]>}.

We then extract the required factor from {\tt G[0]}.

\medskip

\noindent
\newbox\SomeBox
\begin{lrbox}{\SomeBox}
  \begin{minipage}{0.95\textwidth}
\begin{verbatim}
[2071] fctr(G[0]);
[[1,1],[a.re*b.im-b.re*a.im,1],
 [(a.re*b.im-b.re*a.im)*z.re^2+(-ab*b.im*a+bb*a.im*b+b.im-a.im)*z.re
 +(a.re*b.im-b.re*a.im)*z.im^2+(b.re*ab*a-a.re*bb*b+a.re-b.re)*z.im
 -a.re*b.im+b.re*a.im,1]]
\end{verbatim}
  \end{minipage}
\end{lrbox}
\fbox{\usebox{\SomeBox}}

\medskip

From the assumption, the points $a$, $b$ and $0$ are not collinear, 
so {\tt a.re*b.im-b.re*a.im}
$=\mbox{Re}(a)\mbox{Im}(b)-\mbox{Re}(b)\mbox{Im}(a) \neq0 $.
Therefore, the second factor is zero.
As the second factor forms the equation of a circle on $ \mathbb{R}^2 $, 
we need to rewrite it as the equation on $ \mathbb{C}$, as follows.

\medskip

\noindent
\newbox\SomeBox
\begin{lrbox}{\SomeBox}
  \begin{minipage}{0.95\textwidth}
\begin{verbatim}
[2072] Circ=car(@2071[2]);
(a.re*b.im-b.re*a.im)*z.re^2+(-ab*b.im*a+bb*a.im*b+b.im-a.im)*z.re
 +(a.re*b.im-b.re*a.im)*z.im^2+(b.re*ab*a-a.re*bb*b+a.re-b.re)*z.im
 -a.re*b.im+b.re*a.im
[2073] GCircAB=2*@i*subst(Circ,a.re,(a+ab)/2,b.re,(b+bb)/2,
    a.im,(a-ab)/(2*@i),b.im,(b-bb)/(2*@i),z.re,
    (z+zb)/2,z.im,(z-zb)/(2*@i));
((-bb*a+ab*b)*zb+bb*ab*a-bb*ab*b+ab-bb)*z+(((-ab+bb)*b-1)*a+b)*zb
 +bb*a-ab*b
\end{verbatim}
  \end{minipage}
\end{lrbox}
\fbox{\usebox{\SomeBox}}

\medskip

We have now obtained the equation \eqref{eq:g-circle}.

\bigskip

Note that the equation $F(z;a,b,c,d)=0$ 
that determines the $\mbox{GCIS}[a,b,c,d]$ is given by eliminating 
{\tt zb} from {\tt GCircAB} and {\tt GCircCD},
where
\begin{verbatim}
   GCircCD=((-c*db+cb*d)*zb+(cb*c-cb*d-1)*db+cb)*z
            +(d*c*db+(-cb*d-1)*c+d)*zb+c*db-cb*d.
\end{verbatim}
In general, the solution to the equation that eliminates 
$ \mbox{\tt zb}=\overline{z} $ from the equation of two circles 
may not make sense.
For example, for two circles $ |z-2|=1 $ and $ |z-i|=1 $, 
eliminating $\overline{z}$ from the general forms 
$$ z\overline{z}-2z-2\overline{z}+3=0\quad \mbox{and}\quad
   z\overline{z}+iz-i\overline{z}=0 $$
gives the equation 
$$
  (2+i)z^2-(3+4i)z+3i=0.
$$
But, clearly, the solution to this equation satisfies neither $|z-2|=1$
nor $ |z-i|=1$.
This happens because these two circles do not intersect.

Now, we return to the topic of circles on the Riemann sphere.
Two great circles on the Riemann sphere have exactly two 
intersection points unless they coincide.
Thus, the two solutions of $F(z;a,b,c,d)=0$ are guaranteed 
to give the intersection points of the two great circles 
as long as $a, b, c, d $ are not on the same great circle.

\end{document}